\documentclass[11pt]{article}

\usepackage[latin1]{inputenc}
\usepackage[T1]{fontenc}
\usepackage{amsmath,amssymb,amsthm}
\usepackage{enumerate}
\newcommand{\vs}[1][1]{\vspace{#1\baselineskip}}
\usepackage{cite}

 \DeclareMathOperator{\im}{im}
\DeclareMathOperator{\coker}{coker}
 \DeclareMathOperator{\HH}{H}
 
\DeclareMathOperator{\Hom}{Hom}

\DeclareMathOperator{\Pic}{Pic}

\newtheorem{theorem}{Theorem}[section]
\newtheorem{lemma}[theorem]{Lemma}
\newtheorem{corollary}[theorem]{Corollary}

\newtheorem{proposition}[theorem]{Proposition}

\newtheorem{remark}[theorem]{Remark}
\newtheorem{conjecture}[theorem]{Conjecture}

\usepackage{latexsym}            

\newcommand{\pp}{{\mathbb P}}

\newcommand\sF{{\mathcal F}}

\newcommand\sH{{\mathcal H}}
\newcommand\sI{{\mathcal I}}

\newcommand\sL{{\mathcal L}}

\newcommand\sN{{\mathcal N}}
\newcommand\sO{{\mathcal O}}

\def\f{\Gamma}

\def\PP{{\mathbb P}}

\def\ZZ{{\mathbb Z}}

\def\PP{{\mathbb P}}

\newcommand{\proj}[1]
{ \mathchoice
           { {\mathbb P}^{#1} }
           { {\mathbb P}^{#1} }
           { {\mathbb P}^{#1} }
           { {\mathbb P}^{#1} }
         }

\begin{document}
\title{Components of the Hilbert scheme of space curves \\ on
  low-degree smooth surfaces}

\author{Jan O. Kleppe and John C. Ottem}

\date{}
\maketitle
\vspace*{-0.4in}
\begin{abstract}
  \noindent We study maximal families $W$ of the Hilbert scheme, $
  \HH(d,g)_{sc}$, of smooth connected space curves whose general curve $C$
  lies on a smooth surface $S$ of degree $s$. 
  We give conditions on $C$ under which $W$ is a generically smooth component
  of $ \HH(d,g)_{sc}$ and we determine $\dim W$. If $s=4$ and $W$ is an
  irreducible component of $ \HH(d,g)_{sc}$, then the Picard number of $S$ is
  at most $2$ and we explicitly describe, also for $s \ge 5$, non-reduced and
  generically smooth components in the case $\Pic(S)$ is generated by the
  classes of a line and a smooth plane curve of degree $s-1$. For curves on
  smooth cubic surfaces the first author finds new classes of non-reduced
  components of $ \HH(d,g)_{sc}$, thus making progress in proving a conjecture
  for such families. 
  \\[-2.5mm]



 \end{abstract}
 \vspace*{-0.25in}
\thispagestyle{empty}

\section{Introduction and Main Results} 

In this paper we study the Hilbert scheme of smooth connected space curves, $
\HH(d,g)_{sc}$, with regard to dimension and
smoothness, 
with a special emphasis on existence of non-reduced components. The first example of a non-reduced component was found by Mumford \cite{Mu}. There are
several papers that consider such problems, see e.g. \cite{DP}, \cite{E, El,
  F}, \cite{GP2,H1}, \cite{K1, K2,K3,K96, K4, Krao}, \cite{MDP1,MDP2,MDP3,
  N} and the book \cite{H2}.
Here we generalize the approach that was used in \cite{K1} 
for curves on cubic surfaces to study families of curves on smooth surfaces of
degree $s \ge 4$. In particular we investigate when maximal irreducible closed
subsets $W$ of the Hilbert scheme $ \HH(d,g)_{sc}$ whose general curve $C$
lies on a {\it quartic} surface, form non-reduced, or generically smooth,
irreducible components of $\HH(d,g)_{sc}$. We find a pattern similar to what
is known for maximal irreducible families of curves on smooth cubic surfaces;
if $H^1(\sI_C(s))= 0$, $\sI_C$ the sheaf ideal of $C$, 
then $W$
turns out to be a generically smooth component of $
\HH(d,g)_{sc}$. 
If, however, $H^1(\sI_C(s)) \neq 0$ {\it and} the genus is sufficiently large,
then $W$ is still an irreducible component, but it is now non-reduced. For  
$s=4$ it suffices to take ``$g$ large'' as $g > G(d,5)$, the maximum genus of
curves of degree $d$ not contained in a degree-$4$ surface (see
\eqref{maxgen2}), or as
the better bound \\[-1mm]
\begin{equation} \label{ineq} \ g \ > \ \min \{\, G(d,5)-1,\, \frac{d ^{2}
  }{10} + 21 \} \ \ {\rm and} \ \ \ d \geq 21 \, ,
\end{equation}
see Theorem~\ref{main4} and  Corollary~\ref{corrmain4} of Section 4.

Let $s(C)$ denote the minimal degree of a surface containing a curve $C$. If
$W$ is an irreducible closed subset of $ \HH(d,g)_{sc}$, we define
$s(W):=s(C)$ where $C$ is a general curve of $W$. As in \cite{K1} we say $W$
is {\em $s(W)$-maximal} if it is maximal with respect to $s(W)$, i.e. $s(V) >
s(W)$ for any closed irreducible subset $V$ properly containing $W$. We say
$W$ is an $s(W)$-maximal family or subset of $ \HH(d,g)_{sc}$ in this case. By
Remark~\ref{notirredcomp} below, if a very general curve of a 4-maximal family
$W$ sits on a smooth quartic surface $S$ and $d>16$, then the Picard number of
$S$ is {\em at most} 2.

Note that an $s$-maximal family $W$ needs not be an irreducible component of $
\HH(d,g)_{sc}$, but the converse holds (with $s=s(W)$). For instance if $W$ is
a $2$-maximal family whose general curve $C \subset \proj{1} \times \proj{1}$
has bidegree $(p,q)$ , $p \le q$, and degree $d=p+q \ge 6$, then $W$ is not an
irreducible component precisely when $p \le 2$, e.g. if $p=1$ then
$g=(p-1)(q-1)=0$ and the codimension of $W$ in $ \HH(d,g)_{sc}$ is
$4d-(2d+g+8)= 2d-8$, (cf. \cite{T}). Indeed, by Remark~\ref{irredcomp}, 
$g \ge 2d-8$ is a necessary condition for $W$ to be an irreducible component.
This condition is sufficient by \eqref{new4.5} below.

For a $3$-maximal family $W$, $g \ge g_1:= 3d-18$ is necessary (for $d > 9$)
while \cite[Cor.\,17]{K1} shows that $g > g_2:=\lfloor (d^2-4)/8 \rfloor$ is
sufficient for $W$ to be an irreducible component. Indeed $g > g_2$
implies 
\eqref{new4.5} by \cite[Cor.\,17]{K1}. Since the sufficient condition also
implies that $W$ is generically smooth, $W \subset \HH(d,g)_{sc}$ may be a
non-reduced component only when $g_1 \le g \le g_2$. For $g_1=g_2$ and $d >
10$ we have $g=24, d = 14$ 
and the existence of a non-reduced component $W $ with $s(W)=3$ as shown by
Mumford in \cite{Mu}. The first author generalized this result in \cite{K2}
and showed the existence of $3$-maximal families of non-reduced components of
$\HH(d,\lfloor (d^2-4)/8 \rfloor)_{sc}$ for every $d \ge 14$, where $d=14$
corresponds to Mumford's example, see \cite{E}, \cite{K1} and the appendix for
further generalizations and a 
conjecture.

In this paper we consider closely $4$-maximal families $W$ of curves on smooth
quartic surfaces $S \subset \proj{3}$. To get interesting classes, we study
surfaces $S$ where the Picard group $\Pic(S)$ is freely generated over
$\mathbb Z$ by the classes of two smooth connected curves $\Gamma_1$ and
$\Gamma_2$ satisfying $\Gamma_1^2=-2 $, $\Gamma_2^2=0$, $\Gamma_1 \cdot
\Gamma_2=3$, i.e. with
intersection matrix  $ \left(\begin{smallmatrix}-2 & 3 \\
    3 & 0 \end{smallmatrix}\right)$, and such that $H= \Gamma_1+\Gamma_2$ is a
hyperplane section. 
If $C \equiv a\f_1+b\f_2$ are linearly equivalent divisors, we
show that $a,b \ge 0$ if $C$ is a curve (i.e. effective divisor). The
necessary condition $g \ge 4d-33$ of Remark~\ref{irredcomp} for $W$ to be an
irreducible component implies $a > 4$ for $d > 16$ because $g = ad-2a^2+1$. In
Section 5 we prove: 

\begin {theorem} \label{mainQ} Let $S \subset \proj{3}$ be a smooth quartic
  surface with $ \Gamma_1, \Gamma_2$ and $H$ as above, let $C \equiv
  a\f_1+b\f_2$ 
  be a smooth connected curve of degree $d > 16$ 
  and suppose $a \ne b$. Then $C$ belongs to a unique $4$-maximal family $W
  \subseteq \HH(d,g)_{sc}$. Moreover if $\tilde S$ is a quartic surface
  containing a very general member of $W$, then $\Pic(\tilde S)$ is freely
  generated by the classes of a line and a smooth plane cubic curve, and every
  $C \equiv a\f_1+b\f_2$ contained in some surface $S$ as above belongs to
  $W$. Furthermore $\dim W = g+33$, \\[-2mm] $$d=a+3b \ , \ \ \ g=3ab-a^2+1 \
  \ \ and $$
  
  {\rm I)} $W$ is a generically smooth, irreducible component of \
  $\HH(d,g)_{sc}$ provided \\[-1mm]

  \hspace{4.4cm}
    $4  < a  < $  {\Large $ \frac{3b} {2} $} $- 1 $  \hspace{0.3cm} {\rm or}  \hspace{0.4cm} $(a,b)=(5,4).$\\[-2mm]

  {\rm II)} $W$ is a non-reduced irreducible component of \ $\HH(d,g)_{sc}$
  provided
\begin{equation}\label{pic2}
  \frac {3b} {2} - 1 \le a \le \frac {3b}2  \hspace{0.5cm} , \hspace{1.0cm} (a,b) \ne (5,4)
\end{equation} and \eqref{ineq} holds. 
Explicitly, this region is given by the three families
\\[-3mm]
 
\hspace{0.1cm} $a)\,\, (8+3k,6+2k)\, \qquad b)\,\, (10+3k,7+2k) \qquad c)\,\,
(15+3k,10+2k) \text{ for } k \geq 0,$ 
\\[2mm]
and the dimension of their tangent spaces of \  $\HH(d,g)_{sc}$ at $(C)$ is \ 
$$\dim W+ h^1(\sI_C(4))$$ 
where $h^1(\sI_C(4))=1$ {\rm (}resp. $h^1(\sI_C(4))=2$,
$h^1(\sI_C(4))=4$ {\rm )} for the family $a)$  {\rm (}resp. $b),c)$ {\rm )}.
\\[-3mm]
\end{theorem}

One may show that $W$ is non-empty, i.e. that there
exist smooth connected curves $C \equiv a\f_1+b\f_2$ if and only if $0 < a \le
\frac{3b} {2}, \ {\rm or \ } (a,b) \in \{(1,0),(0,1)\}$. 
The case $a=b$ corresponds to $C$ being a
complete intersection of S with some other surface (a c.i. in $S$). \vs[0.3]

We also consider curves sitting on smooth surfaces containing a line and
corresponding $s$-maximal families $W$ for every $s \ge 5$, and we get similar
results as in
I) above while we 
in II) only prove that $W$ is a component (i.e the non-reducedness is open),
cf.\,Theorem~\ref{mainQs}. For $s=5$ we get a little more:

\begin {theorem} \label{mainQ5} Let $S \subset \proj{3}$ be a smooth quintic
  surface containing a line $ \Gamma_1$, let $\Gamma_2 \equiv H - \Gamma_1$,
  $H$ a hyperplane section, be a smooth quartic curve and suppose $\Pic(S)
  \simeq\ZZ \f_1 \oplus \ZZ \f_2$. Let $C \equiv a\f_1+b\f_2$ 
  be a smooth connected curve of degree $d > 25$ 
  with $a \ne b$ and $a,b > 1$. Then $C$ belongs to a unique $5$-maximal
  family $W \subseteq \HH(d,g)_{sc}$. Moreover if $\tilde S$ is a quintic
  surface containing a very general member of $W$, then $\Pic(\tilde S)$ is
  freely generated by the classes of a line and a smooth plane quartic curve,
  and every $C \equiv a\f_1+b\f_2$ contained in some surface $S$ as above
  belongs to $W$. Furthermore $\dim W = -d + g+56$, where \\[-2mm] $$d=a+4b \
  , \ \ \ g=4ab + \frac{1}{2}(a+4b-3a^2)+1 \ \ \ and $$
  
  {\rm I)} $W$ is a generically smooth, irreducible component of \
  $\HH(d,g)_{sc}$ provided 
  {\large  
    $5  < a  <  \frac{4b} {3} - 1 . $} \\[-2mm]

  {\rm II)} $W$ is a non-reduced irreducible component of \ $\HH(d,g)_{sc}$ for
  $(a,b)=(4n,3n)\, , \ n \geq 3$.
\end{theorem}

To get II) we need a natural map $H^0(\sN_C) \to H^1(\sI_C(s))$ to be non-zero,
cf.\,\eqref{alphaN} below. Since this map is surjective for $S$ smooth of
degree $s = 4$, 
$H^1(\sI_C(s)) \ne 0$ suffices to get II) in Theorem~\ref{mainQ}.

Note that we have $0 < a \le \frac{4b} {3}$ for $C$ as above 
and that quintic surfaces as in
Theorem~\ref{mainQ5} exist.

\vspace{2mm}
Another main result (Theorem~\ref{thmcomp}), related to I) above, is obtained
by studying certain maps that involve the relative Picard scheme.
Specializing to $s$-maximal families satisfying $d >s^2$, 
we get:

\begin{theorem} \label{propcomp} Let $s \ge 1$ be an integer and let $W
  \subseteq {\rm H}(d,g)_{sc}$ be an $s$-maximal family such that $d >s^2$.
  Let $C$ be a member of $W$ sitting on a smooth surface $S$ of degree $s$
  satisfying $$H^1(\sI_C(s))= H^1(\sI_C(s-4))=0.$$ Let $E$ be a curve on $S$,
  $H$ a hyperplane section and suppose $C \equiv eE +fH$ for some $e \ne 0,f
  \in \mathbb Z$. Let $t$ be the non-negative integer $ t: = h^1(\sN_{E})-
  h^1(\sO_{E}(s))$ where $\sN_{E}$ is the normal sheaf of $E \subset
  \proj{3}$. If $E$ is either arithmetically Cohen-Macaulay, {\sf or}
  $t=0$ and $H^1(\sI_E(s)) = H^1(\sI_E(s-4))=0$, then $W$ is a generically
  smooth irreducible component of \ ${\rm H}(d,g)_{sc}$ (indeed $C$ and $ eE
  +fH$ are unobstructed), and 
  \begin{equation*} \dim W = (4-s)d+g + {s+3 \choose 3}-2 +
    h^0(\sI_{E}(s-4)) +t \, .
\end{equation*}
\end{theorem}

As a consequence we get a non-trivial formula for $h^1(\sN_C)$. Note that this
theorem shows the unobstructedness of a curve with a multiplicity-$e$
structure on $S$ under some assumptions. The components in
Theorems~\ref{mainQ}, \ref{mainQ5} and \ref{mainQs} correspond to $
h^0(\sI_{E}(s-4))$ large in Theorem~\ref{propcomp},
cf.\,Remark~\ref{remthmcomp}.

Finally the first author consider in the appendix a conjecture about
non-reduced components for maximal families $W \subseteq \HH(d,g)_{sc}$ of
linearly normal curves on a smooth cubic surface $S$ \cite[Conj.\;4]{K1}. In
Theorem~\ref{mainC} he extends the known range where the
conjecture holds. 
We thank O. A. Laudal for interesting discussions on that subject.
We also thank D. Eklund for a discussion of K3 surfaces and R.
Hartshorne for his comments. Also thanks to the careful referee for very 
helpful comments.

\subsection {Notations and terminology} In this paper the ground field $k$ is
{\em algebraically closed of characteristic zero} (and equal to the complex
numbers in the statements where the concept ''very general'' is used).
A surface $S$ in $\proj{3}$ is a hypersurface, defined
by a single equation. A curve $C$ in $\proj{3}$ (resp. in $S$) is a {\it pure
  one-dimensional} subscheme of $\proj{}:=\proj{3}$ (resp. $S$) with ideal
sheaf $\sI_C$ (resp. $\sI_{C/S}$) and normal sheaf $\sN_C =
{\sH}om_{\sO_{\proj{}}}(\sI_C,\sO_C)$ (resp. $\sN_{C/S} =
\Hom_{\sO_{S}}(\sI_{C/S},\sO_C)$). We denote by $d=d(C)$ (resp. $g= g(C)$) the
degree (resp. arithmetic genus) of $C$. If $\sF$ is a coherent
$\sO_{\proj{}}$-Module, we let $H^i(\sF) = H^i(\proj{},\sF)$, $h^i(\sF) = \dim
H^i(\sF)$, $\chi(\sF) = \Sigma (-1)^i h^i(\sF)$ and we often write
$H^i(S,\sO_S(C))$ as $H^i(\sO_S(C))$ for a Cartier divisor $C$ on $S$.
Let 
\[
\ \ \ s(C) = \min \{\, n \, \arrowvert \, h^0(\sI_C(n)) \neq 0 \} \, .
  \]
  We denote by $\HH(d,g)$ (resp. $\HH(d,g)_{sc}$) the Hilbert scheme of (resp.
  smooth connected) space curves of Hilbert polynomial $\chi(\sO_C(t))=dt+1-g$
  \cite{G}. A curve $C$ is called {\it unobstructed} if $\HH(d,g)$ is smooth
  at the corresponding point $(C)$. The curve in a small enough open
  irreducible subset $U$ of $\HH(d,g)$ is called a {\it general} curve of
  $\HH(d,g)$. So any member of $U$ has all the openness properties which we
  want to require. 
A {\it generization} $C' \subset \proj{3}$ of $C \subset \proj{3}$
in 
$\HH(d,g)$ is the general curve of some irreducible subset of $\HH(d,g)$
containing $(C)$. By an irreducible component of $\HH(d,g)$ we always mean a
{\it non-embedded} irreducible component. We denote by ${\rm H}(s)$ the
Hilbert scheme of surfaces of degree $s$ in $\proj{3}$. A member of a closed
irreducible subset $V$ of ${\HH}(s)$ or $\HH(d,g)_{sc}$ is
called {\it very general} in $V$ if it is smooth and sits outside a countable
union of proper closed subset of $V$.

\section{ Background}

In this section we first recall some results from \cite{K1} needed in this
paper. The proofs use the deformation theory developed by Laudal in
\cite{L1}; in particular the results rely on \cite[Thm.\;4.1.14]{L1}.

\subsection {The Hilbert flag scheme} 

Let \label{defofD(d,g;s)} ${\rm D}(d,g;s)$ (resp. ${\rm D}(d,g;s)_{sc}$) be
the Hilbert-flag scheme 
parameterizing pairs $(C,S)$ of curves (resp. smooth connected curves) $C$
contained in surfaces $S$ in $\proj{3}$ with Hilbert polynomials $p(t)=dt+1-g$
and $q(t)= {t+3 \choose 3}- {t-s+3 \choose 3}$ respectively.
Then the tangent space, $A^1:=A^1(C \subset S)$, of ${\rm D}(d,g;s)$ at
$(C,S)$ is given by the Cartesian diagram (i.e. pullback or fibered product);
\begin{equation} \label{bigdia}
\begin{array}{cccccccccccc}
  & &  & \ \ A^1 &  \longrightarrow &  H^0(\sN_S) & \simeq &  H^0(\sO_S(s))
  & & &  \\
  && & \downarrow & \Box &
  \downarrow {m}  &&    \\ 0 \to H^0(\sN_{C/S}) & \rightarrow  & &  H^0(\sN_C)
  & \longrightarrow & H^0(\sN_S\arrowvert_C) &  \simeq &
  H^0(\sO_C(s)) & 
 \end{array}
\end{equation}
where the morphisms are induced by natural (or restriction) maps to normal
sheaves.

Suppose $S$ is a smooth surface of degree $s$. If $C$ is a curve on $S$, we
have $\sN_{C/S} \simeq \omega_C \otimes \omega_S^{-1}$ and a connecting
homomorphism $\delta: H^0(\sN_S\arrowvert_C) \to H^1(\sN_{C/S}) \simeq
H^0(\sO_C(s-4))^{\vee}$ continuing the lower horizontal sequence in
\eqref{bigdia}. Let $\alpha=\alpha_C:=\delta \circ m$ be the composed map and
let $A^2:= \coker \alpha$. Using \eqref{bigdia}, cf. \cite[(2.7) and
Lem.\,8]{K1} for details, we get $\dim A^1-\dim A^2=(4-s)d+g + {s+3 \choose
  3}-2$ and an exact sequence
\begin{equation} \label{alphaN} 0 \to H^0(\sI_{C/S}(s)) \to A^1 \rightarrow
  H^0(\sN_C) \rightarrow H^1(\sI_C(s)) \rightarrow \coker \alpha_C \to
  H^1(\sN_{C}) \rightarrow H^1(\sO_{C}(s))\rightarrow 0 \,.
\end{equation}
The map $ A^1 \rightarrow
H^0(\sN_C)$ in \eqref{bigdia} 
is the tangent map of the $1^{st}$ projection,
\begin{equation} \label{proj1} pr_1: {\rm D}(d,g;s) \longrightarrow {\rm
    H}(d,g)\ , \quad {\rm induced \ by} \quad pr_1((C_1,S_1))= (C_1)\, ,
\end{equation}
at $(C,S)$. Since we may view $ {\rm D}(d,g;s) $ as a relative Hilbert
scheme over ${\rm H}(d,g)$ (cf. \cite[Thm.\;24.7]{H2}), it follows that $pr_1$
is a projective morphism by \cite{G}. By \cite[Lem.\,A10]{K1} $pr_1$ is smooth
at $(C,S)$ under the assumption
\begin{equation} \label{new4.5} H^1(\sI_C(s))= 0\, .
\end{equation} 
Moreover by \cite[(2.6)]{K1} $A^2= \coker \alpha_C$ contains the obstructions
of deforming the pair $(C,S)$, cf. \cite[Thm.\,1.2.7]{K2} for a
detailed version where also the meaning of obstructions is explained.

Let $C$ be a smooth connected curve. If we suppose $d > s^2$ and $s=s(C)$,
then it is easy to see $H^0(\sI_{C/S}(s))=0$ for some hypersurface $S \supset
C$ of degree $s$, and hence, by the semi-continuity of $h^0(\sI_{C}(v))$ for
$v \in \{s-1,s\}$, that {\em the restricted} projection, $pr_1:{\rm
  D}(d,g;s)_{sc} \rightarrow {\rm H}(d,g)_{sc}$, is injective in
$pr_1^{-1}(U)$ for some neighborhood $U \subset {\rm H}(d,g)_{sc}$ of $(C)$.
An $s$-maximal (or just maximal) family $W$ of $ \HH(d,g)_{sc}$ containing
$(C)$ is therefore nothing but 
the image under $pr_1$ of an irreducible component of ${\rm D}(d,g;s)_{sc}$
containing $(C,S)$ (\cite[Def.\;1.24 and Cor.\;1.26]{K3}). If we in addition
suppose that $\alpha$ is surjective, then $(C,S)$ belongs to a unique
generically smooth 
component of ${\rm D}(d,g;s)_{sc}$ and
%
\begin{equation} \label{new4.6} \dim W = h^0(\sN_C) -
  h^1(\sI_C(s)) = (4-s)d+g + {s+3 \choose 3}-2\, .
\end{equation}
Assuming also \eqref{new4.5} it follows that $W$ is a generically
smooth irreducible
component of ${\rm H}(d,g)_{sc}$ (\cite[Thm.\,10]{K1}).

Using the infinitesimal Noether-Lefschetz theorem for $s=4$
(\cite[p.\;253]{GH}) as explained in the proof of \cite[Lem.\,13]{K1}, we
immediately get that $\alpha$ is surjective provided $S$ is smooth of degree
$s \le 4$, $d > s^2$ and $C \subset S$ is smooth and connected, but (for $s=4$
only) not a complete intersection of $S$ with some other surface. 
Hence ${\rm D}(d,g;s)$ is
smooth at $(C,S)$ and we get all conclusions above, assuming \eqref{new4.5}
for the final one. 

\begin{remark} \label{irredcomp} If \ $W $ is an {\sf irreducible component}
  of \ $ {\rm H}(d,g)_{sc}$ containing a curve $C$ sitting on a smooth surface
  $S$ of degree $s:=s(W)$ with $\alpha_C$ surjective and $d > s^2$,
  then 
  $\dim W = (4-s)d+g + {s+3 \choose 3}-2 \ge \chi(\sN_C)=4d$, i.e.
  \begin{equation} \label{gencomp} g \ge sd - {s+3 \choose 3}+2\, ,
  \end{equation} or equivalently, $
  h^1(\sI_C(s)) \le h^1(\sO_C(s))$. Moreover if the general curve of $W$
  does not satisfy \eqref{new4.5}, we get by \eqref{new4.6} that the
  component $W$ is non-reduced (i.e. not generically smooth) and that 
  \eqref{gencomp} holds. 
\end{remark}

\subsection {The relative Picard scheme} 

We also need to consider the Hilbert scheme, ${\HH}(s) \simeq \proj{{s+3
    \choose 3}-1}$, of surfaces of degree $s$ in $\proj{3}$ and the second
projection;
\begin{equation*}  
pr_2: {\rm
  D}(d,g;s) \longrightarrow {\HH}(s) , \quad {\rm induced \ by}
\quad pr_2((C_1,S_1))= (S_1)\, .
\end{equation*} 
Moreover let $\Pic$ be the relative Picard scheme over the open set in
${\HH}(s)$ of smooth surfaces of degree $s$, (see \cite{SB}). Then there is a
projection $p_2: \Pic \to {\HH}(s)$, forgetting the invertible sheaf, and a
rational map,
\begin{equation} \label{pi} \pi: {\rm D}(d,g;s)\ -\!-\!\rightarrow {\Pic} ,
  \quad {\rm induced \ by} \quad \pi((C_1,S_1))= ( \sO_{S_1}(C_1),S_1)
\end{equation} 
which is defined (and denoted $\pi_U$) on the {\it open subscheme} $U \subset
{\rm D}(d,g;s)$ given by pairs $(C_1,S_1)$ where $C_1$ is Cartier on a smooth
$S_1$. Obviously, if we restrict to $U$ we have $p_2 \circ \pi=pr_2$. If $
H^1(S,\sO_{S}(C)) \simeq H^{1}(\proj{3},\sI_{C}(s-4))^{\vee} = 0$ then $ \pi$
is smooth at $(C,S)$ by \cite[Rem.\! 4.5]{SB}. 
Indeed, $ H^1(S,\sL)= 0$, $\sL:= \sO_{S}(C)$, implies a surjective map $A^1
\to T_{\Pic,\sL}$ between the tangent spaces of ${\rm D}(d,g;s)$ at $(C,S)$
and $\Pic$ at $(\sL)$ and an injection $\coker \alpha_{C} \to \coker
\alpha_{\sL}$ on their obstruction spaces (mapping obstructions onto
obstructions), fitting into the following commutative diagram of exact
horizontal sequences 
\begin{equation} \label{tangPic}
\begin{array}{cccccccccccc}
  0 \to H^0(\sN_{C/S})  & \longrightarrow  & \ A^1 &  \longrightarrow &
  H^0(\sN_S) &  \xrightarrow{\,\alpha_{C}\,} &  H^1(\sN_{C/S}) 
  & & &  \\
  \downarrow &   &  \downarrow & &
  \|  &  & \downarrow \\ 0 \to H^1(\sO_{S})= 0 & \longrightarrow  
  &  T_{\Pic,\sL} 
  & \longrightarrow & H^0(\sN_S) & \xrightarrow{\,\alpha_{\sL}\,} & H^2(\sO_S)
  \ .  
 \end{array}
\end{equation}
Here $\alpha_{\sL}$ is the composition of $\alpha_{C}$ with the connecting
homomorphism $H^1(\sN_{C/S}) \to H^2(\sO_S)$ induced from the exact sequence
$0 \to \sO_S \to \sO_S(C) \to\sN_{C/S} \to 0$, cf. \cite[Thm.\,1]{EP},
\cite[Sect.\,4]{K4} and \cite{K98} for some details and compare with
\cite[Ex.\,10.6]{H2} and \cite{Gr}. Indeed, using \cite[Thm.\,1]{Gr} and its
proof we get the following version of the infinitesimal Noether-Lefschetz
theorem (due to Green and Voisin);
\begin{equation} \label{NLs} \dim \im \alpha_{\sL} \ge s-3 \, ,
\end{equation}
making the surjectivity of $\alpha_C$ for $s=4$ mentioned above a special case. In our applications, however, we consider divisors where a basis for $\Pic(S)$ is given, allowing us to compute 
$\dim \coker \alpha_C $ explicitly.

\begin{lemma} \label{lemGrothcomp} Let $S \subset \proj{3}$ be a smooth surface
  of degree $s$, $H$ a hyperplane section, and let $E$ and $C$ be curves on $S$
  satisfying $C \equiv eE +fH$ for some $e \ne 0, f \in \mathbb Z$. Let
  $d=d(C), g=g(C)$ and suppose
  \[H^1(\sI_E(s-4)) = H^1(\sI_C(s-4))=0 \,.\] 
\hspace{4mm} {\rm (i)} Then $ {\rm
    D}(d,g;s)$ is smooth at $(C,S)$ if and only if $ {\rm
    D}(d(E),g(E);s)$ is smooth at $(E,S)$, and
  \[ \dim {\rm D}(d(C),g(C);s)- h^1(\sO_C(s-4))= \dim {\rm D}(d(E),g(E);s)-
  h^1(\sO_E(s-4)) \, ,\] noting that \ $ \dim \arrowvert C \arrowvert
  =h^0(\sO_S(C))-1 = h^1(\sO_C(s-4))$ and \ $ \dim \arrowvert E \arrowvert =
  h^1(\sO_E(s-4))$. Moreover
  \[ \dim \coker \alpha_C + h^0(\sI_{C/S}(s-4)) = \dim \coker \alpha_E +
  h^0(\sI_{E/S}(s-4))\]

  {\rm (ii)} If $H^1(\sI_E(s))=0$ and ${\rm H}(d(E),g(E))_{sc} \ni (E)$ is a
  smooth irreducible scheme, then $ {\rm D}(d,g;s)$ is smooth at $(C,S)$ and
  every $(C',S') \in {\rm D}(d,g;s)$ satisfying $C' \equiv eE' +fH'$ for some
  $(E',S') \in {\rm D}(d(E),g(E);s)_{sc}$, $H'$ a hyperplane section of a {\sf
    smooth} surface $S' \subset \PP^3$, belongs to the unique irreducible
  component of $ {\rm D}(d,g;s)$ containing $(C,S)$.
\end{lemma}

\begin{proof} 
  (i) If $ {\rm D}(d(E),g(E);s)$ is smooth at $(E,S)$, it follows that ${\Pic}$
  is smooth at $(\sO_{S}(E),S)$ by \cite[Rem.\! 4.5]{SB} and
  $H^1(\sI_E(s-4))=0$. Then ${\Pic}$ is smooth at $ (\sO_{S}(C),S)$ because
  the local rings $\sO_{\Pic,(\sO_{S}(C),S)}$ and $\sO_{\Pic,(\sO_{S}(E),S)}$
  are isomorphic, at least up to completion. Indeed, by \cite[Prop.\,2 and
  Constr.\,2]{EP}, $ \alpha_{\sO_{S}(C)} = e \cdot \alpha_{\sO_{S}(E)}$ and
  then \eqref{tangPic} shows that the morphism between the local deformation
  functors of ${\Pic}$ at $(\sO_{S}(E),S)$ and ${\Pic}$ at $(\sO_{S}(C),S)$
  (induced by $\sO_{S}(E) \mapsto \sO_{S}(E)^{\otimes e}(f)$) is an
  isomorphism. Hence $ {\rm D}(d,g;s)$ is smooth at $(C,S)$ by
  $H^1(\sI_C(s-4))=0$, and conversely if $ {\rm D}(d,g;s)$ is smooth at
  $(C,S)$ we get that $ {\rm D}(d(E),g(E);s)$ is smooth at $(E,S)$ by the same
  argument.

  Moreover since the fiber of $\pi$ in \eqref{pi} over $(\sO_{S}(C),S)$ is the
  complete linear system $\arrowvert C \arrowvert$ on $S$ and since smooth
  morphisms have surjective tangent maps, we also get the dimension formulas,
  using duality.

  Finally to determine $\dim \coker \alpha_C$, we use \eqref{tangPic}. Since
  $H^1(\sI_C(s-4))=0$ the map $H^1(\sN_{C/S}) \simeq H^0(\sO_C(s-4))^{\vee}
  \to H^2(\sO_S) \simeq H^0(\sO_S(s-4))^{\vee}$ is injective with cokernel
  $H^0(\sI_{C/S}(s-4))^{\vee}$, whence \[ 0 \longrightarrow \coker \alpha_C
  \longrightarrow \coker \alpha_{\sO_{S}(C)} \longrightarrow
  H^0(\sI_{C/S}(s-4))^{\vee} \longrightarrow 0 \, \] is exact. Since
  $H^1(\sI_E(s-4))=0$ there is a corresponding exact sequence replacing $C$ by
  $E$, and the middle term in these sequences are isomorphic because $
  \alpha_{\sO_{S}(C)} = e \cdot \alpha_{\sO_{S}(E)}$. This implies the final
  dimension formula of (i).

  (ii) By the assumption $H^1(\sI_E(s))=0$, 
  $ {\rm D}(d(E),g(E);s)$ is smooth at $(E,S)$, cf.\,\eqref{new4.5}, whence $
  {\rm D}(d,g;s)$ is smooth at $(C,S)$ by (i). Moreover $ {\rm
    D}(d(E),g(E);s)_{sc}$ is also irreducible since one knows that $pr_1: {\rm
    D}(d(E),g(E);s)_{sc} \to {\rm H}(d(E),g(E))_{sc}$ is irreducible by
  \cite[Thm.\;1.16]{K3}). It follows that the image $U'$ of $\pi': {\rm
    D}(d(E),g(E);s)_{sc} \ \ -\!\rightarrow {\Pic}$ defined as in \eqref{pi}
  is irreducible and since the morphism
\[\eta: U' \to {\Pic} \ \ \ {\rm induced \ by\ } \ \ ( \sO_{S_1}(E_1),S_1)
\mapsto (\sO_{S_1}(E_1)^{\otimes(e)}(f),S_1) \] is smooth (in fact an
isomorphism) onto its image $U'' \subset {\Pic}$ by the argument for their
local deformation functors used in the proof of (i), we get that $U''$ is
irreducible. Finally using that the fiber $\pi_U^{-1}((\sO_{S_1}(C_1),S_1))$
of the morphism in \eqref{pi} is given by the complete linear system
$\arrowvert C_1 \arrowvert$ 
which is irreducible, we get that $\pi_U^{-1}(U'')$ is irreducible (cf. \cite[
Prop.\;1.8]{H3}). Since $(C',S') \in \pi_U^{-1}(U'')$ we are done.
\end{proof}

\begin{remark} \label{notirredcomp} Using \cite[Cor.\;4, p.\;222]{S}, or
  \cite[Prop.\;3.4]{Ek}, and, say, the smoothness of $\pi$ restricted to the
  set $U \cap {\rm D}(d,g;4)_{sc}$ accompanying \eqref{pi}, we get that a
  closed irreducible subset $W$ of \ $ {\rm H}(d,g)_{sc}$, $d > 16$, whose
  very general member $C$ sits on a smooth quartic surface $S$ with Picard
  number $\rho$, will satisfy $\dim W \le g+35-\rho$. Hence if $W \subset {\rm
    H}(d,g)_{sc}$ is 4-maximal (e.g. an irreducible component), then $\rho= 2$
  (or $\rho= 1$ in the c.i. case). 
  \end{remark}

  One should compare Remark~\ref{notirredcomp} with the following result which
  is a special case \cite[Cor.\! II 3.8]{Lop} of a theorem of A. Lopez, see
  \cite[Thm.\! II 3.1]{Lop} for a proof.

  \begin{lemma} \label{lope} Let $E \subset \proj{3}$ be a smooth irreducible
    curve, let $n \ge 4$ be an integer and suppose the degree of every minimal
    generator of the homogeneous ideal of $E$ is at most $n-1$. Let $S$ be a
    very general smooth surface of degree $n$ containing $E$ and let $H$ be a
    hyperplane section. Then $\Pic(S) \simeq \mathbb Z \oplus \mathbb Z$ and
    we may take $\{ \sO_{S}(H), \sO_{S}(E) \}$ as a $ \mathbb Z$-basis for
    $\Pic(S)$.
\end{lemma}

Finally we will need the following lemma to prove our theorems.

\begin{lemma}\label{h1van}
  Let $S$ be a smooth projective surface containing a smooth rational curve
  $\f$ and let $D$ be a divisor such that $c=-D\cdot\f>0$ and $D-\f -K \ne 0$
  is effective, $K$ the canonical divisor.
\begin{itemize}
\item If $H^1(S,\mathcal{O}_S(D-\f))\neq 0$, then $H^1(S,\mathcal{O}_S(D))\neq
  0$.
\item If $c>1$, then $H^1(S,\mathcal{O}_S(D))\neq 0$. In fact, $\dim H^0(S,\mathcal{O}_S(D))\ge c-1$. 
\item If $c=1$ and $H^1(S,\mathcal{O}_S(D-\f))=0$, then
  $H^1(S,\mathcal{O}_S(D))= 0$.
\end{itemize}
\end{lemma}

\begin{proof}
  Taking cohomology of the exact sequence $$0\to \mathcal{O}_S(D-\f)\to
  \mathcal{O}_S(D)\to \mathcal{O}_{\f}(-c)\to 0,$$ and using duality and the
  fact that $\f\simeq \PP^{1}$, we get
$$h^1(\mathcal{O}_S(D))=h^1(\mathcal{O}_S(D-\f))+h^1(\mathcal{O}_{\PP^{1}}(-c))=h^1(\mathcal{O}_S(D-\f))+c-1$$ 
and the result follows.
\end{proof}

\subsection {On the maximum genus of space curves}

Finally we recall the definition of $G(d,s)$; the maximum genus of smooth
connected space curves of degree $d$ not contained in a surface of degree
$s-1$, cf. \cite{GP1}. By
definition, 
\begin{equation} \label{maxgen} G(d,s) = \max \{\, g(C)\, \arrowvert \,
  (C) \in \HH(d,g)_{sc} \ {\rm \ and} \ \ H^0(\sI_{C}(s-1))=0 \, \}
\end{equation}
In the case where $d > s(s-1)$, Gruson and Peskine showed in \cite{GP1} that
\begin{equation} \label{maxgen2} G(d,s) = 1+ \frac d 2 \left( \frac d s +
    s -4 \right) - \frac{r(s-r)(s-1)}{2s} \ \ \ { \rm where } \ d+r \equiv
  0 \ { \rm mod } \ s \ \ {\rm for} \ \ 0 \le r < s,
\end{equation}
and that $g(C)=G(d,s)$ if and only if $C$ is directly linked to a plane curve
of degree $r$ by a c.i. of type $(s,f)$, $f:=(d+r)/s$. Note that this
description of a curve $C$ of $\HH(d,G(d,s))_{sc}$ makes it possible to use
Theorem~\ref{thmcomp} below to find $ \dim V $
where $V \subset \HH(d,G(d,s))_{sc} $ is the irreducible component containing
$(C)$. 
Indeed, we may assume a general member of $V$ is contained in a smooth surface
of degree $s$ because the inequality $r < s$ allows us to start with a smooth
plane curve $E$ of degree $r$ contained in a smooth surface of degree $s$ and
then make a linkage via a c.i. of type $(s,f)$ to get a curve 
$C'$ which, by \cite[Ex.\,3.13]{K3}, belongs to $ V$. Since $C' \equiv fH-E$,
$H$ a hyperplane section, Theorem~\ref{thmcomp} applies and we get $\dim V$
from \eqref{maxgendim}.

\section{A criterion of unobstructedness}
We will now prove a theorem which via Remark~\ref{remthmcomp0} implies
Theorem~\ref{propcomp}. Note that Theorem~\ref{thmcomp} immediately gives us a
formula for $h^1(\sN_C)$ because $C$ is unobstructed, whence $h^1(\sN_C)= \dim
W - 4d$.

\begin{theorem} \label{thmcomp} Let $W'$ be an irreducible component of \
  ${\rm D}(d,g;s)_{sc}$ and let $W:=pr_1(W') \subseteq {\rm H}(d,g)_{sc}$. Let
  $(C,S)$ be a member of $W'$ such that $S$ is smooth of degree $s$ and
  suppose that \[H^1(\sI_C(s)) = H^1(\sI_C(s-4))=0 \,.\] Let $E$ be a curve on
  $S$, $H$ a hyperplane section and suppose $C \equiv eE +fH$ for some $e,f
  \in \mathbb Z$. Let $u:= h^0(\sI_{C/S}(s))+ h^0(\sI_{C/S}(s-4))$ and let $t$
  be the non-negative number $t: = h^1(\sN_{E})- h^1(\sO_{E}(s))$. If $E$ is
  arithmetically Cohen-Macaulay (ACM), or more generally if $E$ is
  unobstructed and satisfies $H^1(\sI_E(s)) = H^1(\sI_E(s-4))=0$, then $W$ is
  a generically smooth irreducible component of \ ${\rm H}(d,g)_{sc}$ (indeed
  $C$ and $ eE +fH$ are unobstructed) of dimension
  \begin{equation} \label{maxgendim} (4-s)d+g + {s+3 \choose 3}-2 - u +
    h^0(\sI_{E/S}(s-4)) +t \
\end{equation} 
\hspace{0.2cm} if \ $e \ne 0$\,; \ \ if \ $e=0$ \ then replace
$h^0(\sI_{E/S}(s-4))+t$ by ${s-1\choose 3}$ in \eqref{maxgendim}. 
\end{theorem}

\begin{remark} \label{remthmcomp0} Let $(C,S)$ and $W$ be as in
  Theorem~\ref{thmcomp} and suppose $d >s^2$. Then $s=s(C)$ and $W$ is an
  $s$-maximal family of ${\rm H}(d,g)_{sc}$ (see paragraph after
  \eqref{new4.5}), in which case $ h^0(\sI_{C/S}(s))$ and
  $h^0(\sI_{C/S}(s-4))$ vanish. Moreover if $E \subset S$ is a curve
  satisfying $h^1(\sN_{E}) = h^1(\sO_{E}(s))$ and $H^1(\sI_E(s)) =0$, then the
  obstruction group $\coker \alpha_E = 0$ by \eqref{alphaN}, whence $E$ is
  unobstructed by \eqref{new4.5}. Noting that $h^0(\sI_{E/S}(s-4))=
  h^0(\sI_E(s-4))$ we get Theorem~\ref{propcomp} from
  Theorem~\ref{thmcomp}. 
\end{remark}

\begin{remark} \label{remthmcomp1} If $E $ is ACM 
  then  $H^1(\sI_{E}(v))=0$ for any $v$ and $E$ is unobstructed by \cite{El}.
\end{remark}

\begin{proof} Firstly we suppose $e \ne 0$. Note that $E$ is unobstructed by
  Remark~\ref{remthmcomp1} and assumption.
  Hence using $H^1(\sI_E(s)) =0$ and the text accompanying \eqref{new4.5} it
  follows that $ {\rm D}(d(E),g(E);s)$ is smooth at $(E,S)$. By
  Lemma~\ref{lemGrothcomp} (i), $ {\rm D}(d,g;s)$ is smooth at $(C,S)$ and at
  $(C',S)$, $C':= eE +fH$. Then the smoothness of $pr_1$, due to
  $H^1(\sI_C(s))\simeq H^1(\sI_{C'}(s))=0$ shows that ${\rm H}(d,g)_{sc}$ is
  smooth at $(C)$ and $(C')$ and that $W$ is a generically smooth irreducible
  component of \ ${\rm H}(d,g)_{sc}$.

  To find $\dim W'= \dim A^1$ and hence $\dim W =\dim W' - h^0(\sI_{C/S}(s))$,
  we need to determine $\dim \coker \alpha_C$. For this we use
  Lemma~\ref{lemGrothcomp} (i) to relate $\dim \coker \alpha_C$ in terms of
  $\dim \coker \alpha_E$. Then using \eqref{alphaN}, we get $\dim \coker
  \alpha_E =t \ge 0$ and hence we obtain the dimension of $W'$ from the
  dimension formula accompanying \eqref{alphaN}. 

  Finally we suppose $e=0$. Then $C$ is a c.i. and it is well known that ${\rm
    H}(d,g)_{sc}$ is smooth at $(C)$ and $(C')$. By the smoothness of $pr_1$,
  we get that $W$ is a generically smooth irreducible component of \ ${\rm
    H}(d,g)_{sc}$. Moreover $H^1(\sN_{C}) \simeq H^1(\sO_{C}(s)) \oplus
  H^1(\sO_C(f))$ and using \eqref{alphaN}, we see that $\coker \alpha_C \simeq
  H^1(\sO_C(f))$ and we conclude the proof by $ H^1(\sO_C(f)) \simeq
  H^0(\sO_C(s-4))^{\vee}$.
\end{proof}
\begin{remark} \label{remthmcomp} {\rm (i)} If $H^1(\sN_{E})=0$,
  e.g.\;$H^1(\sO_{E}(1))=0$ and $E$ reduced, then $E$ is unobstructed and
  $t=0$ in Theorem~\ref{thmcomp}. Observe, however, that in many 
  cases where the unobstructedness of $E$ is known, there is also a dimension
  formula of $h^1(\sN_{E})$, making the number $t$ of Theorem~\ref{thmcomp}
  explicit, see e.g. \cite[Thm.\,1.1]{Krao} for a formula covering both the
  ACM and Buchsbaum diameter-1 case. 

  {\rm (ii)} 
  Assuming $h^0(\sI_{E/S}(s-4))=0$ and $t= 0$ in
  Theorem~\ref{thmcomp} we get generically smooth irreducible components of \
  ${\rm H}(d,g)_{sc}$ equipped with a dimension (``expected dimension'' if $
  h^0(\sI_{C/S}(s))=0$, i.e. $u=0$ according to \cite{K4}) which in \cite{K98}
  was considered to be ``the good general components'' in \ ${\rm
    H}(d,g)_{sc}$ in a certain range of the $d,g$-plane. These components
  correspond, at least infinitesimally, to general components in the
  Noether-Lefschetz locus, see \cite{CiLo} and its references for a discussion
  of this locus, while the components in our main Theorems~\ref{mainQ},
  \ref{mainQ5} and \ref{mainQs} have large $h^0(\sI_{E/S}(s-4))$ and
  correspond to components in the Noether-Lefschetz locus of the smallest
  codimension, see \cite{Gr,Vo, EP}.
 \end{remark}

\section {Irreducible components of ${\rm H}(d,g)_{sc}$}

In the background section we noticed that the assumption $H^1(\sI_C(s))= 0$
for $s = 4$ implies that 4-maximal subsets form generically smooth irreducible
components of ${\rm H}(d,g)_{sc}$. We are now looking for a converse, i.e.
that $H^1(\sI_C(s)) \neq 0$ for $s \le 4$ implies that $s$-maximal subsets
form non-reduced components of ${\rm H}(d,g)_{sc}$. If $s=3$ this is
essentially a conjecture that the first author partially prove in the appendix.
In this section we will see that some ideas of \cite{K1} generalize to cover
cases where $s > 3$ as well. Indeed, we will show the following result which,
together with \eqref{maxgen2}, will be used for proving some results of this
paper.

\begin{theorem} \label{main4} Let $W \subseteq \HH(d,g)_{sc}$ be a $4$-maximal
  family whose general member $C$ is contained in a smooth surface $S \subset
  \proj{3}$ of degree $4$, and suppose that $C$ is not a complete intersection
  of $S$ and some other surface. If
  $h^1(\sI_C(1)) \le d-25$ and \\[-4mm]
  \begin{equation*}  d \geq 31 \ \ \ and \ \ \  g \ > \ 21 +
    \frac{d ^{2} }{10} \  , 
  \end{equation*}
  then $W$ is an irreducible component of \ $\HH(d,g)_{sc}$. Moreover
  $W$ is non-reduced if and only if
 $$H^ 1(\sI_C(4))\neq 0\, . $$
\end{theorem}

\begin{remark} \label{remmain4} 
  Let $C$ be a curve contained in a smooth quartic surface $S$. Using $ \sI_C = \ker
 ( \sO_{ \proj{}} \to \sO_C) $ and the exact sequence $$0 \to \sO_{
    \proj{3}}(-4) \to \sI_C \to \sI_{C/S} \to 0 \ ,$$ we get 
\begin{equation*}
  H^1(\sI_C(4)) \simeq H^1(\sI_{C/S}(4)) \ \ {\rm and} \ \  H^1(\sO_C(4))
  \simeq H^2(\sI_C(4)) \simeq H^2(\sI_{C/S}(4))\, .
\end{equation*}
Moreover since $\sI_{C/S} \simeq \sO_{S}(-C)$, we have by duality
$H^ i(\sI_{C/S}(4))^{\vee} \simeq H^{2-i}(\sO_{S}(C-4H))$, $H$ a
hyperplane section. So to explicitly find non-reduced components given by
Theorem~\ref{main4}, one should look for curves C on $S$ such that $
H^{i}(\sO_{S}(C-4H)) \ne 0$, for $i=0,1$ i.e. such that the linear
system $\arrowvert C-4H \arrowvert$ contains fixed
components, or $C-4H$  is composed with a pencil (cf. \cite{SD}).
\end{remark}

To prove the theorem we will need

\begin {proposition} \label{cliffo} Let $F$ be an integral surface in
  $\proj{3}$ of degree $s \geq 4$, let $S \to F$ be a desingularization and
  let $C \subset F$ be a smooth connected curve of degree $d$ and genus $g$
  such that $Sing(F) \cap C$ is a finite set. If $\sN_{C/S}$ is the normal
  sheaf of $C \hookrightarrow S$ (i.e. of the proper transform of $C
  \hookrightarrow F$)
  and ${\rm Hilb}(F)_{sc}$ is the Hilbert scheme of smooth connected curves on $F$, then \\[-3mm]
\begin{equation*}
  \dim_{(C)} {\rm Hilb}(F)_{sc} \ \le \ \dim\,H^0(\sN_{C/S}) \ \le \  \max \,\{
  \,{\frac{d ^{2} }{s}}   -   g   
  + 1   ,  {\frac{d ^{2} }{2s}}  + 1\,\}\, .
\end{equation*}
\end{proposition}

\begin{proof} The first inequality is  \cite[Lemma 22]{K1}. In \cite[Lemma
  23]{K1} we use  Hodge's index theorem to show that $\deg \sN_{C/S} = C^2 \le
  d^2/s$. Therefore \ $H^1(\sN_{C/S}) = 0$ implies  $h^0(\sN_{C/S}) = \chi
  (\sN_{C/S}) = d^2/s +1 -g$ by Riemann-Roch while if $H^1(\sN_{C/S}) \ne 0$
  then  Clifford's theorem gives \ $\dim \arrowvert \sN_{C/S} \arrowvert \le
   \frac {1}{2} \deg \sN_{C/S} \le d^2/2s$, i.e.  $h^0(\sN_{C/S}) \le d^2/2s+1$.
\end{proof}


In the proposition below we extend \cite[Prop.\;20]{K1}, which should have
assumed ``$Sing(F) \cap C$ finite" or $d > (s(C)-1)^2$, from $s=4$ to $s \ge
4$.

\begin {proposition} \label{20} Let V be an irreducible component of
  $\HH(d,g)_{sc}$ whose general curve $C$ sits on some integral surface $F$ of
  degree $s \geq 4$. If 
  $d > s^2$, then
\begin{equation*}
  \dim\, V  \le \  {s+3 \choose 3}- 1 + \max \,\{ \,{\frac{d ^{2} }{s}}-g\,
  , \ {\frac{d ^{2} }{2s}}\, , \ (4-s)d+g-1+h^0(\sO_C(s-4)) \, \}\, .
\end{equation*}
\end{proposition}

\begin{proof} Let $W$ be any irreducible component of $ {\rm D}(d,g;s)_{sc}$
  containing $(C,F)$. Since the $2^{nd}$ projection, $pr_2: {\rm
    D}(d,g;s)_{sc} \rightarrow \ {\HH}(s)$ has the Hilbert scheme ${\rm
    Hilb}(F)_{sc}$ as its fiber over $(F)$, it follows that
\begin{equation} \label{dimpr2}
 \dim W \ \le \ \dim pr_2(W) + \dim_{(C)} {\rm Hilb}(F)_{sc}
\end{equation} 
where $pr_2(W)$ is the scheme theoretic image of $ pr_2$ restricted to $W$.
Indeed endowing $W$ with its reduced induced scheme structure, we may look
upon the induced map $pr_2': W \to pr_2(W)$ as a morphism between integral
schemes whose fiber over $(F)$ is at least contained in ${\rm Hilb}(F)_{sc}$. We
get \eqref{dimpr2} by the fact that the dimension of every component in a
fiber of $pr_2'$ over $(F)$ is not smaller than the relative dimension of
$pr_2'$, cf. \cite[Ch.\,II, Ex.\,3.22]{H1}.

Suppose $F$ is {\em smooth}. Then $ \dim pr_2(W) \le \dim_{(F)} \HH(s)={s+3
  \choose 3}- 1$.
Moreover, 
$\sN_{C/F} \simeq \omega_C \otimes \omega_F^{-1}$ leads to
$ \chi(\sN_{C/F}) = \chi( \omega_C(4-s)) = (4-s)d+g-1 $ and 
\begin{equation*} \label{smoo}
 \dim_{(C)} {\rm Hilb}(F)_{sc} \le h^0(\sN_{C/F}) = (4-s)d+g-1 + h^0(\sO_C(s-4))\, .
\end{equation*}
Suppose $F$ is {\em not smooth, but integral}, then $pr_2$ is at least
non-dominating, whence $ \dim pr_2(W) \le {s+3 \choose 3}- 2$. To use
Proposition~\ref{cliffo} to bound $ \dim_{(C)} {\rm Hilb}(F)_{sc}$, we must show
that $Sing(F)\cap C$ is a finite set. Indeed if this set is not finite, then
the smooth connected curve $C$ is contained in $Sing(F)$ which implies $d \le
(s-1)^2$ because there is a c.i. of type $(s-1,s-1)$ containing $C$
(chosen among the partial derivatives of the form
defining $F$). This contradicts an assumption of Proposition~\ref{20} while
the other assumptions imply the existence of an irreducible component $W \ni
(C, F)$ of ${\rm D}(d,g;s)_{sc}$ which dominates $V$ under the first
projection $pr_1$ given in \eqref{proj1}.
Since $d > s^2$ then 
$\dim V = \dim W$ and we can use \eqref{dimpr2} and the upper bounds of $
\dim_{(C)} {\rm Hilb}(F)_{sc}$ to get Proposition~\ref{20}.
\end{proof}


\begin{proof}[Proof (of  Theorem~\ref{main4})]
  
  To see that $W$ is an irreducible component, we suppose there exists a
  component $V$ of $\HH(d,g)_{sc}$ satisfying $W \subset V$ and $\dim W < \dim
  V$. Then $s:=s(V) \ge 5$ by the definition of a 4-maximal family. Moreover
  $s=5$ since the case $s \geq 6$ can be excluded. Indeed using \eqref{maxgen}
  and \eqref{maxgen2} we get $g > G(d,6)$ from $21+d^2/10 > 1+d^2/12+d$ and
  the assumptions of the theorem. To get a contradiction we will use
  Proposition~\ref{20} for $s = 5$, and \eqref{new4.6} that implies $\dim W
  =g+33$. Indeed $\alpha_C$ is surjective by the infinitesimal
  Noether-Lefschetz theorem and by assumption, see the paragraph before
  Remark~\ref{irredcomp}, and we get \eqref{new4.6}. Let $C'$ be the general
  curve of $V$. Then $s(C')= 5$ and $C'$ is a smooth connected curve. It
  follows that a surface $F'$ containing $C'$ of the least possible degree,
  namely 5, is integral. 
We get \\[-2mm]
\[ g+33 < 55 + \max \,\{ \, \lfloor {\frac{d ^{2} }{5}}\rfloor -g\, , \
\lfloor {\frac{d ^{2} }{10}} \rfloor \, , \ -d+g-1+4+h^1(\sI_{C'}(1)) \,
\}\, . \] Suppose the maximum to the right is obtained by $\lfloor d ^{2}/5
\rfloor -g$. Then since $ g+33 < 55 + d ^{2}/5-g$ is equivalent to $ g < 11 +
d ^{2}/10$, we get a contradiction to the displayed assumption of the theorem.
Similarly, $ g+33 < 55 + \lfloor d ^{2}/10 \rfloor $ will lead to a
contradiction. Finally if we suppose $$g+33 < 55 -d+g-1+4+h^1(\sI_{C'}(1))
\, ,$$ i.e. $h^1(\sI_{C'}(1)) > d-25$ and we use that $h^1(\sI_{C'}(1)) \le
h^1(\sI_{C}(1))$ by semi-continuity, we get $h^1(\sI_{C}(1)) > d-25$ which
again is a contradiction to the assumptions. Thus we have proved that $W$ is
an irreducible component of $\HH(d,g)_{sc}$.

Then using \eqref{new4.6}, i.e. $ \dim W + h^1(\sI_C(4)) = h^0(\sN_C)$,
it is straightforward to get the final statement of the theorem, and we
are done.
\end{proof}
\begin{corollary} \label{corrmain4} With notations and assumptions as in the
  first sentence in Theorem~\ref{main4}, suppose in addition
\begin{equation*} \label{ineqcor} g \ > \ \min \{\, G(d,5)-1,\, \frac{d ^{2}
  }{10}  + 21 \} \ \ {\rm and} \ \ \ d \geq 21 \, .
\end{equation*}
Then $W$ is an irreducible component of \ $\HH(d,g)_{sc}$, and $W$ is
non-reduced if and only if $H^ 1(\sI_C(4))\neq 0. $
\end{corollary}
\begin{proof} One checks that the minimum value in the corollary is equal to
  $G(d,5)-1$ (resp. $21+d^2/10$) for $21\le d \le 44$ (resp. $d \ge 45$).
  Moreover, $h^1(\sI_{C}(1)) = h^0(\sO_{C}(1))-4 \le \max \{d-g, \frac d 2 \}
  -3$ for a non-plane curve by Clifford's theorem and Riemann-Roch. Hence if
  $d \ge 45$, we get $h^1(\sI_{C}(1)) \le d -25$ and we conclude by
  Theorem~\ref{main4}.

  If \ $21 \le d \le 44$ we suppose there is an irreducible component $V$ of
  $\HH(d,g)_{sc}$ satisfying $W \subset V$ and $\dim W < \dim V$. We may
  suppose either $g > G(d,5)$ or $g=G(d,5)$. In the first case we get $s(V)=4$
  which contradicts the 4-maximality of $W$. In the remaining case, using
  \eqref{maxgen2} and the arguments after \eqref{maxgen2}, we may assume the
  general curve $C'$ of $V$ is linked to a smooth plane curve $E$ of degree $r
  < 5$ by a c.i. of type $(5,(d+r)/5)$ where the quintic surface $S$ in the
  c.i. is smooth. Hence we can apply Theorem~\ref{thmcomp} (replacing $W$ in
  Theorem~\ref{thmcomp} by $V$) with $s=5$ and $e=-1$ (and $e=0$ when $r=0$)
  to compute $\dim V$. Since $\sN_E \simeq \sO_E(1) \oplus \sO_E(r)$, whence
  $H^1(\sN_E) \simeq H^1(\sO_E(1))$ we easily compute $ h^0(\sI_{E/S}(1)) +t$
  in Theorem~\ref{thmcomp} to be $4 - \chi(\sO_E(1)) \le 2$. We get $ \dim V
  \le 56-d+g - h^0(\sI_{C'/S}(5))$ (resp. $ \dim V \le 58-d+g$ for $r=0$) by
  Theorem~\ref{thmcomp}. In particular we have $ \dim V \le 58-d+g$ for $d \ge
  25$ which contradicts $g+33=\dim W < \dim V$. For $21 \le d \le 24$ we have
  by \eqref{maxgen2} at least two quintic surfaces in the c.i. $Y$ linking
  $C'$ to $E$. Thus $h^0(\sI_{C'/S}(5)) \ge 1$. In the case $d=21$ we get
  $h^0(\sI_{C'/S}(5))= 2$ because $\omega_E \simeq \sI_{C'/Y}(6)$ imply
  $h^0(\sI_{C'/Y}(5))= h^0(\sO_E)=1$. Hence $\dim V \le
  56-d+g-h^0(\sI_{C/S}(5))$ implies $\dim V \le g+33$, and we have a
  contradiction. Finally using the left equality of \eqref{new4.6} we easily
  get the statement on non-reducedness of the corollary, and we are done.
\end{proof}
\begin{remark} 
  Let $C$ be a general curve of a $3$-maximal family
  $W$. 
  The analogue of Theorem~\ref{main4} for $s(C)=3$ states that $W$ is an
  irreducible 
  component of \ $\HH(d,g)_{sc}$ provided $$g > 7+ (d-2)^2/8 \ {\rm and} \ d
  \ge 27 $$
  as one may easily deduce from the proof of \cite[Theorem 5]{K1}, paying a
  little extra attention to the case
  $(d,g)=(30,106)$. 
  To show that $W$ is a non-reduced irreducible component, the above result
  turned out to be quite useful in \cite{K1}. This result, together with
  Theorem~\ref{main4} for $s(C)=4$, improve upon what we may show by only
  using \eqref{maxgen2} by $k+d/2$, $k$ a constant, cf.
  Corollary~\ref{corrmain4}. This improvement is not necessary for
  Theorem~\ref{mainQ} of this paper because the curves in {\rm II)} satisfy $g
  > G(d,5)-1$. We need, however, Theorem~\ref{main4} in the appendix, and we
  hope it, or a refined version, applies to other classes of components where
  $s(C) = 4$, as it did for $s(C) =
  3$. 
\end{remark} 

\section{Components of ${\rm H}(d,g)_{sc}$ for $s=4$}
In this section we prove Theorem~\ref{mainQ} stated in the
introduction. Let us start by considering the existence of the quartic
surfaces that we need in the sequel, together with determining the smooth connected curves contained in the
surfaces. 

%

\subsection{Quartic surfaces containing a line}\label{Pic2}
Our main example comes from studying curves on smooth quartic surfaces containing a line. Such quartic surfaces appeared in the work of Mori \cite{Mor},
who showed the following result: If there exists a smooth quartic surface
$S_0$ containing a nonsingular curve $\f_0$ of degree $d$ and genus $g$, then
there also exists a smooth quartic surface $S$ containing a smooth curve $\f$
of the same degree and genus, such that $\Pic(S) \simeq \ZZ \f \oplus \ZZ H$,
where $H$ is the hyperplane section. (See also \cite[p. 138]{H2}).

The following result describes the main properties of curves on such quartics (see also \cite[Section 3.1]{ottem}).

\begin{proposition}\label{moriq}
  There exists a smooth quartic surface $S\subset \PP^3$ with $H \equiv
  \f_1 + \f_2$ and $\Pic(S) \simeq \ZZ\f_1 \oplus \ZZ\f_2$ where $\f_1,\f_2$
  are smooth curves of genus 0 and 1 respectively, and intersection matrix
  given by
$$ \left(
 \begin{array}{cc}
   \f_1^{2} & \f_1\cdot \f_2  \\
   \f_1\cdot \f_2 & \f_2^{2}
 \end{array}
 \right) =
\begin{pmatrix}
-2 & 3 \\
3 & 0
\end{pmatrix} \ .
$$Furthermore, for any such surface $S$ with $H,\f_1, \f_2$ as above the following hold:
\begin{enumerate}[i)]
\item Any effective divisor class can be written as $a\f_1+b\f_2$ for
  non-negative integers $a,b\ge 0$.
\item A divisor class $a\f_1+b\f_2$ is nef if and only if $3b\ge 2a\ge 0$.
\item If $D\equiv a\f_1+b\f_2$ is a divisor with $3b\ge 2a> 0$, then the
  general element in $|D|$ is a smooth irreducible curve. Conversely, the
  classes of the irreducible divisors correspond to classes $a\f_1+b\f_2$ with
  $3b\ge2a> 0$ or $(a,b)=(1,0),(0,1)$.
\end{enumerate}
\end{proposition}


\begin{proof}
  Smooth quartic surfaces $S_0$ containing a line $\{x_0=x_1=0\}$ are defined
  by a homogeneous polynomial of the form $F=x_0 p+x_1q=0$ where $p,q\in
  k[x_0,x_1,x_2,x_3]$ are cubic forms. By Mori's result above there exists a
  smooth quartic surface $S$ such that $\Pic(S)$ is generated by the classes
  of a smooth rational curve $\f_1$ and the hyperplane section $H$. By the
  adjunction formula, we have $\f_1^2=-2$. In fact the diophantine equation
  $(xH+y\f_1)^2=4x^2+2xy-2y^2=2(2 x - y) (x + y)=-2$ has the only solutions
  $(0,\pm 1)$, showing that $\f_1$ is the unique $(-2)$-curve on $S$. The
  class $H-\f_1$ has self-intersection $0$ and is thus effective. It is in
  fact represented by the smooth elliptic curve given by
  $\{x_0=q=0\}$. 

  To prove $i)$ we claim that every effective divisor is linearly equivalent
  to a non-negative integral linear combination of $\f_1$ and $\f_2$. Indeed,
  let $D$ be any effective divisor class and write $D=a\f_1+b\f_2$ for
  integers $a,b$. We may assume that $D\cdot \f_1 \ge 0$ (otherwise $\f_1$ is
  a fixed component of the linear system $|D|$ and we can instead consider
  $D-\f_1$). Then we have $0\le D\cdot\f_1=3b-2a$ and $0\le D\cdot \f_2=3a$
  implying that $a,b\ge 0$. Dually we have also shown that the nef cone is
  determined by the inequalities $a\ge 0$ and $3b\ge 2a$, giving $i)$ and
  $ii)$.

  $iii)$: If $C$ is an irreducible curve with $C\neq \f_1,\f_2$, then $C$ is nef
  and $C\cdot \f_2>0$ (by the Hodge index theorem). So $C\equiv a\f_1+b\f_2$
  with $3b-2a\ge 0$ and $a> 0$. Conversely, if these conditions are satisfied,
  the divisor $D=a\f_1+b\f_2$ is base-point free \cite[Corollary 3.2]{SD} and
  hence by Bertini's theorem the general element in $|D|$ is smooth and
  irreducible. 
\end{proof}

For the existence of such a K3 surface one could also use a result of
Nikulin \cite{Nik} which states that for any even lattice of signature
$(1,\rho-1)$ with $\rho\le 10$, there exists a smooth projective K3 surface
with this intersection form. Using this, and the embedding criteria of
Saint-Donat \cite{SD}, one can show that any surface with intersection matrix
as above embeds as a smooth quartic
surface. 

\bigskip

To prove Theorem~\ref{mainQ} 
it is necessary to determine  when $H^1(S,\mathcal{O}_S(D))\neq 0$. 
If $D=a\f_1+b\f_2$ is effective, then by the proposition above, we must have
$a,b\ge 0$. If $a=0$ then
$h^1(S,\mathcal{O}_S(D))=h^1(S,\mathcal{O}_S(b\f_2))=\max\{b-1,0\}$. We will
assume $a>0$. If now $c:=-D\cdot \f_1=2a-3b\le 0$, then $D$ is nef and
$D^2=a(3b-2a)+3ab>0$ and so $h^1(S,\mathcal{O}_S(D))=0$. If $c=1$, then $a>1$
(due to $1=2a-3b$) and $D-\f_1$ is nef with $(D-\f_1)^2>0$ and so
$H^1(\mathcal{O}_S(D-\f_1))=0$ and consequently $H^1(S,\mathcal{O}_S(D))=0$ by
Lemma \ref{h1van}. The same lemma implies $H^1(S,\mathcal{O}_S(D))\neq 0$ for
$D \ne \f_1$ and $c>1$. Hence we obtain

\begin{proposition}Let $S$ be a smooth quartic surface with $\Pic(S) \simeq
  \ZZ\f_1\oplus \ZZ\f_2$ and $\f_1,\f_2$ and $H$ as above and suppose
  $D=a\f_1+b\f_2$ is an effective divisor class with $a>0$. Then for $D \ne
  \f_1$,
\begin{equation} \label{2a3b}
H^1(S,\mathcal{O}_S(D))\neq 0 \qquad {\rm if \ and \ only \ if} \qquad 2a>3b+1.
\end{equation}Moreover  $H^1(S,\mathcal{O}_S(\f_1))= 0$ and if $a=0$, then $h^1(S,\mathcal{O}_S(D))=\max\{b-1,0\}$. 
\end{proposition}

\begin{proof}[Proof of Theorem~\ref{mainQ}]  
  We get $d=a+3b$, $g=3ab-a^2+1$ from from $d=C \cdot H$, $g=1+C^2/2$ and
  since $C \notin \arrowvert nH \arrowvert$ for every $n \in \mathbb Z$ by
  assumption, it follows that $C$ is not a c.i. in $S$. By \cite[Thm.\! 10 and
  Lem.\;13]{K1}, see Section 2.1 of this paper, the Hilbert-flag scheme ${\rm
    D}(d,g;s)$ for $s=4$ is smooth at $(C,S)$ of dimension $\dim A^1=g + 33$.
  Hence $(C,S)$ belongs to a unique irreducible component of ${\rm
    D}(d,g;4)_{sc}$ whose image under the $1^{st}$ projection, $pr_1: {\rm
    D}(d,g;4)_{sc} \rightarrow {\rm H}(d,g)_{sc}$, is 
  the 4-maximal subset $W$ of the theorem because the assumption $d > 16$
  implies 
  $s(W)=4$.

  By Lemma~\ref{lemGrothcomp} we have the properties of $W$ stated in
  Theorem~\ref{mainQ} provided we can take a $\mathbb Z$-basis of $\Pic(\tilde
  S)$ as in the theorem. Since we may assume that $\tilde S$ is very general
  by using the projections $pr_i$, $i=1,2$ and the very general assumption
  concerning $W$, this is straightforward by Lemma~\ref{lope}. Indeed there is
  a hyperplane section $\tilde H$ of $\tilde S$ containing $\tilde E$ such
  that $\f:=\tilde H- \tilde E$ is a smooth curve of degree 3 (\cite{SD}) and
  instead of the basis $\{\sO_S(\tilde E),\sO_S(\tilde H) \}$ given by
  Lemma~\ref{lope}, we may take the classes of $\{\tilde E,\f \}$ as a
  $\mathbb Z$-basis of $\Pic(\tilde S)$.

For the rest of the proof we use \eqref{2a3b}, \eqref{ineq}
as in Corollary~\ref{corrmain4}, 
and
Theorem~\ref{propcomp}.

I) By Theorem~\ref{propcomp} it suffices to show that $
H^{1}(S,\sO_S(C-4H))=0$ 
because $h^1(I_C(4))= h^{1}(\sO_S(C-4H))$. But this is immediate from
\eqref{2a3b} since the inequality $4 < a < \frac{3b} {2} - 1$ implies $a-4 >0$
and $2(a-4) \le 3(b-4)+1$, and the case $(a,b)=(5,4)$ corresponds to
$C-4H=\f_1$.

II) By Corollary~\ref{corrmain4} it suffices to show that
$H^1(S,\sO_S(C-4H))\neq 0$ for the classes in (\ref{pic2}) of
Theorem~\ref{mainQ}. Also this is immediate from \eqref{2a3b}, since
(\ref{pic2}) and $d > 16$ imply $2(a-4)>3(b-4)+1$, $a-4 >0$ and $(a-4,b-4) \ne
(1,0)$.
The lattice points
in this region satisfying (1) are then found by inspection.

For the statement on the dimension of the tangent space of
$\HH(d,g)_{sc}$ at $(C)$, we know that this dimension is equal to $g+33+
h^1(I_C(4))$ by \eqref{new4.6}. We claim that $h^1(I_C(4))=1$ (resp.
$h^1(I_C(4))=2$, $h^1(I_C(4))=4$) for the family $a)$ (resp. $b),c)$). To see
it we use the short exact sequence in the proof of Lemma \ref{h1van} for $\f:=
\f_1$ and $D:=C-4H$ (and then to $D:=C-4H-\f_1$ for the class b)). Taking
cohomology 
and counting dimensions, we get the claim.
\end{proof}

\begin{remark} Note that if $D:=C-4H=a\f_1+b\f_2$ is effective and
  $h^1(S,\sO_S(D))>0$, then either $\f_1$ is a fixed component of $|D|$ or $D$
  is composed with a pencil, in which case $C=4\f_1+r\f_2$ for $r\ge 6$. In
  the latter case, it is easy to verify that $C$ does not satisfy the
  constraints \eqref{ineq}, hence does not lead to non-reduced components by the
  theory we have so far.
%
\end{remark}

\subsection{Other quartic surfaces }

The surface appearing in Theorem \ref{mainQ} is an example of a quartic
surface for which we can use the theory of this paper to describe the
smoothness properties of the components of the Hilbert scheme. In fact, by the
result of Mori quoted above, it is clear that there should exist many such
examples, but finding ones with irreducible curves satisfying the bound
\eqref{ineq} seems more difficult. Nevertheless, let us finish our study of
curves on a smooth quartic by giving the main details of one more class.

Consider a homogeneous quartic form of the form  $F=x_0p+q_1q_2$ where $q_1,q_2$ are quadrics defining the plane conics and $p$ is a cubic. For $q_1,q_2$  general, $F$ defines a smooth quartic surface $S_0 \subset \PP^3$, where the hyperplane section splits into two plane conics $(\{x_0=q_1=0\}$ and $(\{x_0=q_2=0\} $). Then by Mori's theorem, for $q_1,q_2$  very general, one obtains a smooth quartic surface with the intersection matrix $(\f_i\cdot\f_j)=\left(\begin{smallmatrix}-2 & 4\\
    4 & -2 \end{smallmatrix}\right)$. (Again we choose the basis
$\{\sO_S(\f_1),\sO_S(\f_2)\}$ for $\Pic(S)$ rather than
$\{\sO_S(H),\sO_S(\f_1)\}$.) As before, one can show that $\f_1,\f_2$ define
smooth irreducible $(-2)$-curves which generate the semigroup of effective
divisors.

    
\begin{proposition} \label{othercl} There exists a smooth quartic surface $S$
  with $\Pic(S) \simeq \ZZ\f_1\oplus \ZZ\f_2$ and $\f_1,\f_2$, $H\equiv \f_1+\f_2$ as
  above.
\begin{enumerate}[i)]
\item Any effective divisor class can be written as $a\f_1+b\f_2$ for
  non-negative integers $a,b\ge 0$.
\item A divisor class $a\f_1+b\f_2$ is nef if and only if $ \frac {b} {2} \le
  a \le 2b$.
\item If $D\equiv a\f_1+b\f_2$ is a divisor with $a,b>0$, then
  $H^1(S,\sO_S(D))=0$ if and only if $D$ is nef.
\item If $D\equiv a\f_1+b\f_2$ is a divisor with $0<\frac {b} {2} \le a \le
  2b$, then the general element in $|D|$ is a smooth irreducible curve.
  Conversely, the classes of the irreducible curves correspond to classes
  $a\f_1+b\f_2$ satisfying $a,b>0$ and $\frac {b} {2} \le a \le 2b$ or
  $(a,b)=(1,0),(0,1)$.
\end{enumerate}
  \end{proposition}\begin{proof}
    The first part of the proposition and $i)$ follow as in the
    proof of Proposition \ref{moriq}. If $D=a\f_1+b\f_2$ is nef and non-zero,
    then intersecting with $\f_1$ and $\f_2$ gives the above inequality for
    $ii)$. In $iii)$, if $D$ is nef and $\not \equiv 0$, we have
    $D^2=8ab-2a^2-2b^2= 2a(2b-a)+2b(2a-b)>0$ and so by the
    Kawamata-Viehweg vanishing theorem, $H^1(S,\sO_S(D))=0$. Conversely, if
    $D$ is not nef, then we can without loss of generality assume
    $d=-D.\f_1>0$. But $d$ must be an even number, hence $d>1$ and so
    $H^1(S,\sO_S(D))\neq 0$ by Lemma \ref{h1van}.
\end{proof}
Note that Proposition~\ref{othercl} (iii) allows us to see exactly when
$h^1(\sI_C(4))=h^1(S,\sO_S(C-4H))=0$, 
and we get at least:
\begin {proposition} \label{mainQ2} Let $S \subset \proj{3}$ be a smooth
  quartic surface with $ \Gamma_1, \Gamma_2 $ as above, let $C \equiv
  a\f_1+b\f_2$ be a smooth connected curve and suppose $a \ne b$ and $d > 16$.
  Then $C$ belongs to a unique $4$-maximal family $W \subseteq \HH(d,g)_{sc}$.
  Moreover if $\tilde S$ is a quartic surface containing a very general member
  of $W$, then $\Pic(\tilde S)$ is freely generated by the classes of two
  rational conics, and every $C \equiv a\f_1+b\f_2$ contained in some surface
  $S$ as above belongs to $W$. Furthermore $\dim W = g+33$, \\[-3mm] $$
  \hspace{1.2cm} d=2a+2b \ , \ \ \ g=4ab-a^2-b^2+1 \ \ \ and $$
  
$W$ is a generically smooth, irreducible component of \
  $\HH(d,g)_{sc}$ provided \
  {\large $  \frac {b} {2} +2 \le a  \le  2b - 4 . $} \\[-2mm]


\end{proposition}
%

\begin{proof} 
The proof of the properties of the
  maximal family $W$ follows as in the first part of Theorem~\ref{mainQ}. The
  remaining part follows from Theorem~\ref{propcomp} and
  Proposition~\ref{othercl} since $H^1(S,\sO_S(C-4H))= 0$ if and only if
  $C-4H$ is nef if and only if $\frac{b+4}2\le a \le 2(b-2)$.
\end{proof}

\begin{remark} \label{Q2class} {\rm (i)} We may by symmetry restrict the range
  of Proposition~\ref{mainQ2} to $a > b$. Then there are 4 families in the
  range $ 2b-4 < a \le 2b$
  which satisfy $H^1(\sI_C(4)) \neq 0$. 
  They are of the form $(5+2k,4+k)$ $(6+2k,4+k)$ $(7+2k,4+k)$ $(8+2k,4+k)$, $k
  \ge 1$.
  Unfortunately, \eqref{ineq} does not hold for any of these classes, so we
  can not conclude that they correspond to non-reduced components by the
  results we have so far. We {\sf expect}, however, that they are non-reduced
  components. 

 {\rm (ii)} We are informed by H. Nasu 
 that he, using methods appearing in \cite{N, MuNa,Na} is able to show that
 the family $(5+2k,4+k)$ corresponds to non-reduced components of ${\rm
   H}(d,g)_{sc}$, and that these methods also apply to show the non-reducedness
 of family a) of Theorem~\ref{mainQ}, cf. Remark~\ref{MNa}.
\end{remark}

\section{Components of ${\rm H}(d,g)_{sc}$ for $s=5$}

Let $S$ be a very general smooth surface of degree 5 in $\PP^3$ defined by an
equation $x_0P+x_1Q$, where $P,Q$ are very general homogeneous degree-4
polynomials. Let $\f_1=\{x_0=x_1=0\}$ (a line) and $\f_2=\{x_0=Q=0\}$ (a plane
quartic). The hyperplane section $H\in|\mathcal{O}_{\PP^3}(1)|_S|$ satisfies
$H\equiv K \equiv \f_1+\f_2$ and $H^2=5$ where $K$ is the canonical divisor,
and we may suppose $\Pic(S) \simeq\ZZ \f_1 \oplus \ZZ \f_2$ by
Lemma~\ref{lope}. Then $\f_1\cdot H=1$, $H^2=5$ and the adjunction formula
imply that
the intersection matrix is $(\f_i\cdot\f_j)=\left(\begin{smallmatrix}-3 & 4\\
    4 & 0 \end{smallmatrix}\right)$. 
    
    Let $C\subset S$ be a smooth, connected curve of degree $d$ and genus $g$ with $C\equiv a\f_1+b\f_2$. We have $d=C \cdot H$,
$g=1+(C^2+C\cdot K)/2$,  and we deduce $$d=a+4b \ \ {\rm and} \ \ g=1+4ab +
\frac{1}{2}(a+4b-3a^2)\, .$$

As in the case of quartic surfaces, we easily deduce  the following result:
\begin{lemma} \label{lemnef5}
Any effective divisor on $S$ is linearly equivalent to $a\f_1+b\f_2$ where $a,b\ge 0$. Every nef divisor is linearly equivalent to  $a\f_1+b\f_2$ where $4b\ge 3a\ge 0$. 
\end{lemma}

\def\O{\mathcal{O}}

It will be of interest to study the divisor $4\f_1+3\f_2$, which is on the boundary of the nef cone.

\begin{lemma}\label{43bpf}
Let $S$ be a quintic surface with $\f_1,\f_2$ as above. Then the divisor $D=4\f_1+3\f_2$ is base-point free. Moreover, for each $m\ge 1$, $|mD|$ contains a smooth irreducible curve.
\end{lemma}

\begin{proof}
  Choose global sections $x$ and $y_1,y_2$ as bases of $H^0(S,\sO_S(\f_1))$
  and $H^0(S,\sO_S(\f_2))$ respectively. Note that as $\f_2$ is base-point
  free, so is the linear system $V=\langle
  y_1^3,y_1^2y_2,y_1y_2^2,y_2^3\rangle\subseteq H^0(S,\sO_S(3\f_2))$. Note
  that $x^4\cdot V=\langle
  x^4y_1^3,x^4y_1^2y_2,x^4y_1y_2^2,x^4y_2^3\rangle\subseteq H^0(S,\sO_S(D))$,
  so if $D$ has a base-point, it is necessarily contained in $\f_1$. But a
  general divisor $M\in |D|$ does not intersect $\f_1$: In particular this is
  true for the curve $M=\{P=Q=0\}$, for $P,Q$ general. The last part now
  follows from Bertini's theorem since $mD$ is not composed with a pencil.
\end{proof}

\begin{lemma}\label{h1vanish5}
Let $C$ be a general element of the linear system $|a\f_1+b\f_2|$, where $4b\ge 3a$ and $a > 1$. Then $C$ is a smooth irreducible curve with $H^i(S,\mathcal{O}_S(C))=0$ for $i>0$.
\end{lemma}

\begin{proof}
  Since the inequality $4b\ge 3a \ge 0$ describes exactly the nef cone of $S$,
  i.e., $C\cdot \f_i\ge 0$ for $i=1,2$, it follows that $C$ is a nef divisor.
  Assume first that $C$ is ample, i.e., that $4b>3a> 0$, then 
  $C-K=C-\f_1-\f_2$ is nef, and big by $a > 1$ which implies $b>1$, and
  so by Kawamata-Viehweg,
  $H^i(S,\mathcal{O}_S(C))=H^i(S,\mathcal{O}_S(K+(C-K)))=0$ for $i>0$.
  Moreover, in this case, Bertini's theorem gives that the general element is
  smooth and irreducible, since $|C|$ is base-point free.

It remains to consider the case $C\equiv 4m\f_1+3m\f_2$. It was shown above
that $C$ is smooth and irreducible. Let $D\in | 4\f_1+3\f_2|$ be a general
smooth element. Then we get $H^1(S,\O_S(D))=0$ by the exact sequence
$$
0\to \O_S(3\f_1+3\f_2)\to \O_S(D)\to \O_{\f_1}\to 0 \, .
$$
Moreover, there is also an exact sequence
$$
0\to \O_S((m-1)D)\to \O_S(mD)\to \O_S(mD)|_D\to 0 \, .
$$By induction on $m$, $H^1(S,\O_S((m-1)D))=0$. Also, a computation gives that
$\O_S(mD)|_D$ has degree $>2g(D)-2$ on $D$ for $m \ge 2$, so $H^1(D,
\O_S(mD)|_D)=0$. Hence $H^i(S,\mathcal{O}_S(mD))=0$ for $i>0$.
\end{proof}

\begin{proof}[Proof of Theorem~\ref{mainQ5}] 


  If $E:=\f_1$, we get $H^1(\sI_E(v))=0$ for any $v$, whence ${\rm D}(1,0;5)$
  is smooth at $(E,S)$ by \eqref{new4.5}. Then Lemma~\ref{lemGrothcomp} and
  \eqref{alphaN} imply that the Hilbert-flag scheme ${\rm D}(d,g;5)_{sc}$ is
  smooth of dimension $\dim A^1=\dim \coker \alpha_C -d+g + 54$ at $(C,S)$
  because $ H^1(\sI_C(1)) \simeq H^1(\sO_S(C))^{\vee} = 0\, $ by Lemma~\ref{h1vanish5}, since $C$ is smooth and irreducible. 
  Hence $(C,S)$ belongs to a unique irreducible component of ${\rm
    D}(d,g;5)_{sc}$ whose image under the $1^{st}$ projection, $pr_1: {\rm
    D}(d,g;5)_{sc} \rightarrow {\rm H}(d,g)_{sc}$, is 
  a $5$-maximal subset $W$ because the assumption $d > 25$ implies $s(W)=5$.
  This $W$ is the one given in the theorem. Thanks to Lemma~\ref{lemGrothcomp}
  we get the properties of $W$ stated in Theorem~\ref{mainQ5}. Also $\dim
  \coker \alpha_C = 2$ is easily found using Lemma~\ref{lemGrothcomp}, and we
  get $ \dim W=-d+g + 56$.

  Now to get I) it suffices by \eqref{new4.5} to show $ H^1(\sI_C(5))=0$.
  Since $ H^1(\sI_C(5))^{\vee} \simeq H^1(\sO_S(C+K-5H))= H^1(\sO_S(C-4H))\, $
  this group vanishes by Lemma~\ref{h1vanish5}: Indeed $C-4H \equiv
  (a-4)\f_1+(b-4)\f_2$ satisfies $4(b-4)\ge 3(a-4)$ and $a-4>1$ by the
  assumption $5 < a < \frac{4b} {3} - 1$ of I).

  To get II) we will show $g-G(d,6) > 0$ where $ G(d,6) = 1+ d + d^2/12 -
  5r(6-r)/{12}$ and $0 \le r < 6$ are given by \eqref{maxgen2}. Since $g=
  1+8n+24n^2$ and $d=16n$ it is straightforward to get $g - (1+ d + d^2/12)=
  8n(n-3)/3$, whence $g-G(d,6) > 0$ for $n > 3$ and $g=G(d,6)$ for $n=3$. Then
  we can use exactly the arguments in the $2^{nd}$ paragraph of the proof of
  Corollary~\ref{corrmain4} to see that $W$ is an irreducible component of $
  \HH(d,g)_{sc}$, i.e. we only need to show that $\dim W \ge \dim V$ for the
  irreducible component $V$ of $\HH(d,g)_{sc}$ containing a curve $C'$
  of maximum genus: $g=G(d,6)=241$ where $d=48$. By \eqref{maxgen2} the curve
  $C'$ is a c.i. of type $(6,8)$, whence with dualizing sheaf $\omega_{C'}
  \simeq \sO_{C'}(10)$. Then we conclude by $\dim
  V=4d+h^1(\sO_{C'}(6))+h^1(\sO_{C'}(8))=237$ and $\dim W =-d+g+56=249$.

  This component is non-reduced if we can show $\dim W < h^0(\sN_C)$ for $C$
  general. Since $\dim W = \dim A^1$, $H^0(\sI_{C/S}(5)) = 0$ and $\dim \coker
  \alpha_C = 2$ it suffices by \eqref{alphaN} to prove $h^1(\sI_C(5)) \ge 3$.
  This follows from the exact sequence in the proof of Lemma~\ref{h1van},
  because $\f_1 \simeq \proj{1}$.
\end{proof}

\section{Components of ${\rm H}(d,g)_{sc}$ for $s \ge 5$}

Let, as in the case of quintic surfaces, $S$ be a very general smooth surface
of degree $s \ge 5$ in $\PP^3$ defined by an equation $x_0P+x_1Q$, where $P,Q$
are very general homogeneous polynomials of degree $s-1$. Let
$\f_1=\{x_0=x_1=0\}$ and $\f_2=\{x_0=Q=0\}$. The hyperplane section satisfies
$H \equiv \f_1+\f_2$, $H^2=s$ and we may suppose $\Pic(S) \simeq\ZZ \f \oplus
\ZZ H$ by Lemma~\ref{lope}. If $C \equiv a\f_1+b\f_2$ then $d=C \cdot H$,
$K=(s-4)H$ and the adjunction formula imply that the intersection matrix is
$(\f_i\cdot\f_j)=\left(\begin{smallmatrix}2-s & s-1\\
    s-1 & 0 \end{smallmatrix}\right)$, and
that 
\begin{equation} \label{sge5}
 d=a+(s-1)b \ \ {\rm and} \ \ g=1+(s-1)ab +
\frac{1}{2}((s-4)a+(s-4)(s-1)b-(s-2)a^2)\, .
\end {equation}
The first two lemmas of Section 6 generalize easily and we get
\begin{lemma} \label{nefbp}
Any effective divisor on $S$ is linearly equivalent to $a\f_1+b\f_2$ where
$a,b\ge 0$. Every nef divisor is linearly equivalent to  $a\f_1+b\f_2$ where
$(s-1)b\ge (s-2)a\ge 0$.  
\end{lemma}

%

\begin{lemma}\label{s43bpf}
  Let $S$ be a smooth surface of degree $s$ with $\f_1,\f_2$ as above. Then
  the divisor $D=(s-1)\f_1+(s-2)\f_2$ is base-point free and $|mD|$ contains a
  smooth irreducible curve for each $m\ge 1$. Moreover if $C \equiv
  a\f_1+b\f_2$ is 
  any divisor satisfying $C \cdot \f_1 > 0$ and $a >1$ then $|C|$ contains a smooth irreducible curve.
\end{lemma}

Indeed, for the final sentence we remark that $C-H$ is nef, hence base-point
free since it is a linear combination of base-point free divisors,
$\f_2,H,(s-1)\f_1+(s-2)\f_2$, with non-negative coefficients. 

Even though it seems that we do
not get a result similar to Lemma~\ref{h1vanish5} in full generality, we can
at least use the first paragraph of its proof to get
\begin{equation} \label{KVvan} H^i(S,\mathcal{O}_S(C))=0 \ \ {\rm for} \ i>0 \
  \  {\rm provided} \  \ a>s-4 \  {\rm and} \ (s-1)b \ge (s-2)a+s-4 
\end{equation}
because the assumptions on $a,b$ imply that $C-K$ is nef and big. From this, we are led to

\begin {theorem} \label{mainQs} Let $S \subset \proj{3}$ be a smooth
  degree-$s$ surface containing a line $ \Gamma_1$, let $\Gamma_2 \equiv H -
  \Gamma_1$ 
  be a smooth curve and suppose $\Pic(S) \simeq\ZZ \f_1 \oplus \ZZ \f_2$ and  $s \ge 5$. Let
  $C \equiv a\f_1+b\f_2$ 
  be a smooth connected curve of degree $d >
  s^2$ 
  with  $a \ne b$.

\item {\rm (i)} Suppose $ a>s-4$ and $(s-1)b \ge (s-2)a+s-4$. Then $C$ belongs
  to a unique $s$-maximal family $W \subseteq \HH(d,g)_{sc}$. Moreover if
  $\tilde S$ is a degree-$s$ surface containing a very general member of $W$,
  then $\Pic(\tilde S)$ is freely generated by the classes of a line and a
  smooth plane degree-$(s-1)$ curve, and every $C \equiv a\f_1+b\f_2$
  contained in 
  some surface $S$ as above belongs to $W$. Furthermore
  \[\dim W =(4-s)d+g + {s+3 \choose 3}+ {s-1 \choose 3} -s+ 1 \ \ with \ d,g \
  as \ in \ \eqref{sge5}, \]  
  and if $(a,b) \ne (2s-2,2s-4)$ for $s=5,6$, then $W$ is an
  irreducible component of \ $\HH(d,g)_{sc}$.
\item {\rm (ii)} Suppose $s < a < $ {\Large$\frac{(s-1)(b-1)}{s-2}$}. Then all
  conclusions of {\rm (i)} hold and $W$ is a generically smooth irreducible
  component  of \ $\HH(d,g)_{sc}$.
\end{theorem}

\begin{proof}

The proof follows the proof of the first part and I) of
Theorem~\ref{mainQ5} except for $W$ being an irreducible component in (i). Let us go through the main points. 

The assumptions on $a,b$ in (i) imply that
$H^1(\sO_S(C))=0$ by the Kawamata-Viehweg vanishing theorem \eqref{KVvan}. Moreover, if we replace $(a,b)$ in (i) by $(a-4,b-4)$, we get exactly the assumptions of (ii), leading also to $H^1(\sO_S(C-4H))=0$.
Given the first vanishing, we now use Lemma~\ref{lemGrothcomp} to get the
stated properties of $W$ in (i). We also get $$ \dim \coker \alpha_C =
h^0(\sI_{E/S}(s-4)= h^0(\sO_{\proj{3}}(s-4))- h^0(\sO_{E}(s-4)) \, .$$ Indeed
since $E$ is a line we have $ \coker \alpha_E =0$ by \eqref{alphaN}. It
follows that $\dim W= \dim A^1(C \subset S)$ is given by the formula
accompanying \eqref{alphaN} recalling $ A^2 \simeq \coker \alpha_C$, i.e.
$\dim W$ is as stated. Finally having both vanishings, we also get (ii) using
the smoothness of $pr_1$ at $(C,S)$, cf. \eqref{new4.5}.

It remains to prove that $W$ is an irreducible component also when
$H^1(\sO_S(C-4H))\ne 0$, e.g. to show that $g>G(d,s+1)$, and in the
case $g=G(d,s+1)$ to show $\dim V \le \dim W$ for the component $V \subset
H(d,G(d,s+1))_{sc}$ of curves of maximum genus mentioned in subsection 2.3.
Noticing that \begin{equation} \label{maxgen5} g-G(d,s+1) = \frac d 2 \left(
    2a-1 - \frac {d}{s+1} \right)- \frac {a^2s}{2} + \epsilon
\end{equation} by \eqref{maxgen2} and \eqref{sge5} where $ \epsilon= \frac{r(s+1-r)s}{2(s+1)} \ $  and  $ r$ is given by $\ d+r \equiv 0 \ { \rm mod } \ (s+1) \ \ {\rm for} \ \ 0 \le r \le s$, 
we first consider curves on the boundary $(s-1)b=(s-2)a$, i.e. where $C
\equiv n(s-1)\f_1 +n(s-2)\f_2$.
Inserting $a=n(s-1)$, $d=n(s-1)^2$ and $a^2=nd$ into \eqref{maxgen5} and
denoting $ \epsilon(C):= \epsilon$ we get \begin{equation*}
  g-G(d,s+1) = \frac {d} {2(s+1)} \left(n(s-3)- (s+1) \right) +
  \epsilon(C) \ . \end{equation*} It follows easily that $g\ge
G(d,s+1)$ except in the cases $s=5,6$ and $(a,b)=(2s-2,2s-4)$, and moreover
that $g = G(d,s+1)$ only for $(s,n) \in \{(7,2),(5,3) \}$. The case $(s,n)
=(5,3)$ yields $d=48$ and $g=241$, cf. proof of Theorem~\ref{mainQ5}.
We get $\dim V=237 < \dim W$. The case $(s,n) =(7,2)$ is similar. Indeed let
$V$ be the component of $\HH(d,g)_{sc}$ containing a curve $C'$ of
maximum genus $G(d,6)=469$ where $d=72$. By \eqref{maxgen2} $C'$ is a c.i. of type $(8,9)$ with dualizing sheaf $\omega_{C'} \simeq
\sO_{C'}(13)$. We get $\dim V=4d+h^1(\sO_{C'}(8))+h^1(\sO_{C'}(9))=379$, while
$\dim W =-3d+g+134=387$ by Theorem~\ref{mainQs}, i.e. $W$ is an irreducible
component.

Finally, it suffices to show $g(D)>G(d(D),s+1)$ for $D \in |C+kH|$, $k > 0$
and $C$ on the mentioned boundary. Using $d(D)=d+ks$ and \eqref{maxgen5}, or
directly $g(D)=g+kd+sk(s-4+k)/2$, we prove that
\begin{equation} \label{maxgen6} g(D)-G(d(D),s+1)-[ g-G(d,s+1)] = \frac
  {k(2d+sk)}{2(s+1)}- \frac {sk}{2} + \epsilon(D)- \epsilon(C) \
  .\end{equation} To
compute 
$\epsilon(D)- \epsilon(C)$, we remark that $\epsilon(C)$
does not change using the number $\rho$ satisfying $\ d \equiv \rho \ { \rm
  mod } \ (s+1) \ \ {\rm for} \ \ 0 < \rho \le s+1$ instead of $r$ in
$\epsilon(C)$ because $r=s+1-\rho$. Since $d(D)=d+ks \equiv \rho-k \ {
  \rm mod } \ (s+1)$, we have (where the inequality is an equality if $\rho-k
> 0$)
$$\epsilon(D)- \epsilon(C)\ge[(\rho-k)(s+1-\rho+k)s- \rho(s+1-\rho)s]/(2s+2)=-k(s+1-2\rho+k)s/(2s+2).$$
Combining with \eqref{maxgen6} we get
\begin{equation*} \frac {k(2d+sk)}{2(s+1)}- \frac {sk}{2} +
  \epsilon(D)- \epsilon(C)= \frac {k}{s+1}(d -s(s+1)+\rho s) >
  0 \end{equation*} because $d>s^2$ and $\rho \ge 1$. This shows
$g(D)>G(d(D),s+1)$ for $(s,n) \notin \{(5,2),(6,2)\}$. In the cases $(s,n) \in
\{(5,2),(6,2)\}$, a direct computation show $(d -s(s+1)+\rho s)/(s+1) = 2$ and
$ g-G(d,s+1)=-2$ (resp. $-1$) for $s=5$ (resp. $6$). Hence $g(D)>G(d(D),s+1)$
except in the case $(s,n,k)=(5,2,1)$ where we have $g(D)=G(d(D),6)=150$ and
$d(D)=37$, $D=C+H$. Since $\dim W= -d(D)+g(D)+56=169$ by Theorem~\ref{mainQs}
while Theorem~\ref{propcomp} with $s=6$ (replacing $W$ in
Theorem~\ref{propcomp} by $V$) implies $\dim V = -2d(D)+g(D)+82+4+3=165 < \dim
W$, $W$ is also now an irreducible component and we are done.
\end{proof}
\begin{remark} \label{MNa} The components $W$ of Theorem~\ref{mainQs} may be
  non-reduced components of \ ${\rm H}(d,g)_{sc}$ in a range close to the
  boundary of the nef cone even though we only succeed to prove it for $s=5$
  and $(a,b)=(4n,3n)$, $n\ge 3$ (Theorem~\ref{mainQ5}). If we had been able to
  compute $h^0(\sN_C)$ for a general curve $C$ of $W$, we could conclude that
  all components satisfying $\dim W < h^0(\sN_C)$ were non-reduced. Another
  promising approach to deal with this problem is to compute the cup-product
  of an element of $H^0(\sN_C)$ and show that it is non-zero, as Mukai and
  Nasu do in \cite{N, MuNa,Na} by using the role of effective divisors with
  negative self-intersection in linear systems corresponding to $\arrowvert
  C-4H \arrowvert$.
\end{remark}
\section{Appendix on non-reduced components of ${\rm H}(d,g)_{sc}$ for $s=3$}

{  by Jan O. Kleppe}

\bigskip 
In this section we look at progress to the conjecture below. Note that
a maximal family $W$ is closed and irreducible by our definition, and that $\dim
W = d+g+18$ always holds provided $d > 9$.

\begin{conjecture} \label{conj} Let $W$ be a $3$-maximal 
  family of smooth connected, {\it linearly normal} space curves of degree $d
  > 9$ and genus $g$, whose general member $C$ is contained in a smooth cubic
  surface. Then $W$ is a non-reduced irreducible component of $ \HH(d,g)_{sc}$
  if and only if
 $$ d \geq 14, \ \ 3d - 18  \leq  g \leq  (d^2-4)/8 \ \ and  \ \
H^1(\sI_C(3))\neq 0 \, .$$
 \end{conjecture}

 This conjecture, originating in \cite{K2}, is here presented by modifications
 proposed by Ellia \cite{E} (see also \cite{DP} by Dolcetti, Pareshi), because
 they found counterexamples which heavily depended on the fact the general
 curves were {\it not} linearly normal (i.e. the curves satisfied
 $H^1(\sI_C(1)) \neq 0$).

 The conjecture is known to be true in many cases. Indeed Mumford's well known
 example (\cite{Mu}) of a non-reduced component is in the range of
 Conjecture~\ref{conj} (minimal with respect to both degree and genus).
 Also the main result by the author in \cite{K1} shows that the conjecture
 holds provided $g > 7 + (d-2)^2/8$, $ d \geq 18$, and Ellia makes further
 progresses in \cite{E} which we comment on later. Recently Nasu proves (and
 reproves) a part of the conjecture by showing that the cup-product
\begin{equation*} \label{cup} \ \ H^0(\mathcal{N}_{C}) \times
  H^0( \mathcal{N}_{C}) \rightarrow H^1( \mathcal{N}_{C})
\end{equation*} 
is nonzero if $h^1(\sI_C(3))=1$ (\cite{N}). In this section we will see that
the methods of \cite{K1} and a nice result of Ellia in \cite{E} imply that we
can greatly enlarge the range where Conjecture~\ref{conj} holds.
 
Now recall that a smooth cubic surface $S$ is obtained by blowing up
$\proj{2}$ in six general points (see \cite{H1} and \cite{GP2}). Taking the
linear equivalence classes of the inverse image of a line in $\proj{2}$ and
$-E_i$ (minus the exceptional divisors), $i = 1,..,6$, as a basis for
$\Pic(S)$, we can associate a curve $C$ on $S$ and its corresponding
invertible sheaf $\sO_S(C)$ with a $7$-tuple of  integers
$(\delta,m_1,..,m_6)$ satisfying
\begin{equation} \label{numpic}
  \delta \geq m_1 \geq.. \geq m_6 \ \ and \ \ \delta \geq  m_1 + m_2 +m_3 \, .
\end{equation}
The degree and the (arithmetic) genus of the curve are given by

\begin{equation*}
d  = 3\, \delta - \sum_{i=1 }^{6} \, m _{i}   ~ ,  ~   ~  g   =  {
  \delta -1  \choose 2}   -   \sum_{i=1 }^{6} \,{ m _{i} \choose 2}  \, .
\end{equation*}
In terms of a $7$-tuple $(\delta ,m_1,..,m_6)$ satisfying \eqref{numpic} one
may use Kodaira vanishing theorem and a further analysis (see \cite[Lem.\;16
and Cor.\;17]{K1}) to verify the following facts for a curve
$C$;\\

(A) If $m_6 \geq 3$ and $(\delta,m_1,..,m_6) \neq
(\lambda+9,\lambda+3,3,..3)$ for any $\lambda$ $\geq$ 2, then $
\HH^1(\sI_C(3)) = 0$. In particular if a curve  on a smooth
cubic satisfies $g > (d^2-4)/8 $, then
\begin{equation*}
 H^1(\sI_C(3)) = 0 \, .
\end{equation*} 

(B) If $m_6 \geq 1$ and $(\delta, m_1,..,m_6) \neq
(\lambda+3,\lambda+1,1,..1)$ for any $\lambda \geq 2$, then $H^1(\sI_C(1)) =
0$. Moreover, in the range $ d \geq 14$ and $g \geq 3d-18$, we have

  $$ H^1(\sI_C(3)) \neq 0 \ {\rm and} \ H^1(\sI_{C}(1))= 0 \ \ \ {\rm if \
    and \ only \ if } \ \ \ \ 1 \leq m_6 \leq 2 \, .$$

  \begin{remark} i) The explicit size of the interval where
    $H_{*}^1(\sI_{C}):=\oplus_v H^1(\sI_{C}(v))$ is non-vanishing (and a proof
    of it) was originally found by Peskine and Gruson (see
    \cite[Prop.\;3.1.3]{K2}). 

    ii) The case \ $m_6 = 0$ is treated by Dolcetti and Pareshi in \cite{DP}.
    In this case they found a range in the $(d,g)$-plane where the maximal
    subsets $W$ were {\sf contained} in a non-reduced component of dimension >
    $d+g+18$, see also \cite[Rem. VI.6]{E}.
\end{remark}

Using (A) and the fact that $H^1(\sI_C(3)) = 0$ implies unobstructedness and
$\dim W = d+g+18$ (for $d > 9$), one may easily see that the conditions of
Conjecture~\ref{conj} are necessary for $W$ to be a non-reduced component. The
conjecture therefore really deals with the converse, and we may suppose $m_6 =
1$ or $ 2$ by (B). For both values the main theorem of this section shows that
the conjecture is true under weak assumptions, thus generalizing the main
results of \cite{E} and \cite{K1} to:

\begin {theorem} \label{mainC} Let $W$ be a 3-maximal
  family of smooth connected space curves, whose general member sits on a
  smooth cubic surface $S$ and corresponds to the $7$-tuple
  $(\delta,m_1,..,m_6)$, $\delta \geq m_1
  \geq .. \geq m_6$ and $\delta \geq m_1 + m_2 +m_3$, of $\Pic(S)$. 
Then  \\[-2mm]
  
i) $W$ is a generically smooth, irreducible
component of $\HH(d,g)_{sc}$ provided  \\[-2mm]

\hspace{2.5cm} $m_6 \geq 3$ and $(\delta, m_1,..,m_6) \neq
(\lambda+9,\lambda+3,3,..3)$ for any $\lambda \geq 2$, \\[-2mm]

ii) $W$ is a non-reduced irreducible component of $\HH(d,g)_{sc}$
provided; \\[-2mm]

\hspace{0.2cm} a) \ $m_6=2 ,\ m_5 \geq 4 , \  d \geq 21$ \ and \ $(\delta,
 m_1,..,m_6) \neq (\lambda+12,\lambda+4,4,..,4,2)$ for any $\lambda \geq 2$, or
\\[-2mm]

\hspace{0.2cm} b)  \ $m_6 = 1, \ m_5 \geq 6, \ d \geq 35$ \
and \ $(\delta, m_1,..,m_6) \neq (\lambda+18,\lambda+6,6,..,6,1)$ for any
$\lambda \geq 2$,
or  \\[-2mm]

\hspace{0.2cm} c) \ $m_6 = 1, m_5 = 5, m_ 4 \geq 7, \ d
\geq 35$ \ and \ $(\delta,
 m_1,..,m_6) \neq (\lambda+21,\lambda+7,7,..,7,5,1)$ for $\lambda \geq 2$.
\end{theorem}

In the exceptional case $(\lambda+9,\lambda+3,3,..,3)$ of $i)$ 
we have $H^1(\sO_C(3)) = 0$; whence $W$ is contained in a unique generically
smooth irreducible component $V$ of $ \HH(d,g)_{sc}$ and $\dim V - \dim W =
h^1(\sI_C(3))$ (cf. \cite[Thm.\;1]{K1}). For $m_6=2$ in $ii)$ Nasu's result in
\cite{N} gives a better range, see Remark~\ref{oldPreprint} ii).

To prove Theorem ~\ref{mainC}, we will need the following two results:

\begin{proposition} \label{Ellia} (Ellia) Let $d$ and $g$ be integers such
  that $ d \geq 21$ and $g \geq 3d - 18$, let $W$ be as in Theorem
  ~\ref{mainC} and suppose the general curve $C$ of $W$ satisfies $
 H^1(\sI_C(1)) = 0$. If $C'$ is a generization of $C$ in $\HH(d,g)_{sc}$
  satisfying $H^0(\sI_{C'}(3)) = 0$, then $H^0(\sI_{C'}(4)) = 0$.
\end{proposition}

\begin{proof}  See \cite[Prop. VI.2]{E}. 
\end{proof}

We remark that Ellia uses this key proposition to prove the conjecture
provided $d \geq 21$ and $g > G(d,5)$, cf. \eqref{maxgen}. His result is in
most cases clearly better than the result in \cite{K1} which requires $g > 7 +
(d-2)^2/8$, $ d \geq 18$, because $G(d,5) = d^2/10 + d/2 + \epsilon$,
$\epsilon$ a correction term. 
There are, however, many cases where Theorem~\ref{mainC} implies the
conjecture while Ellia's result does not.

\begin{lemma} \label{lem1} Let $C$ be a curve sitting on a smooth cubic
  surface $S$, whose corresponding invertible sheaf is given by
  $(\delta,m_1,..,m_6)$, $\delta \geq m_1 \geq .. \geq m_6$ and $\delta \geq
  m_1 + m_2 +m_3$. If $v$ is a non-negative integer such that $m_3 \ge v$, and
  $(\delta, m_1,..,m_6) \neq (\lambda+3v,\lambda+v,v,v,m_4,m_5,m_6)$ for any
  $\lambda \geq 2$, then
\begin{displaymath} 
  h ^{0} (\sI_{C} \, (v)) - h ^{1} (\sI_{C} \, (v))   \ge  {v \choose 3}   -  
  \sum_{m_i < v} \,{v+1 - m_i \choose 2} 
\end{displaymath}
where the sum is taken among those $i \in \{4,5,6\}$ satisfying $m_i < v$.
\end{lemma}

\begin{proof} 
  Let $b_i:= \max\{0,m_i-v\}$ and notice that the invertible sheaf $\sL$,
  given by $(\delta-3v,b_1,..,b_6)$, is generated by global sections because
  $b_6 \ge 0$ and $\delta -3v \geq b_1+b_2+b_3$ (cf. \cite[Sect.\;2]{GP2}).
  Moreover $(\delta-3v,b_1,..,b_6) \ne (\lambda,\lambda,0,..,0)$ for $\lambda
  \ge 2$ by assumption; whence $H^0(\sL)$ contains a smooth connected curve
  $\overline D$ (take $\overline D=0$ in the special case
  $(\delta-3v,b_1,..,b_6)= (\lambda,\lambda,0,..,0)$ with $\lambda=0$).

  Let $n_i:= - \min\{0,m_i-v\}$ for $i \in \{4,5,6\}$, let $F:= \sum n_iE_i$
  and observe that $D := \overline D+F$ is an effective divisor (or zero) of
  the linear system $\arrowvert C-vH \arrowvert$ corresponding to
  $(\delta-3v,m_1-v,..,m_6-v)$. By e.g. the algorithm of \cite[Rem.\;2.7]{Gi},
  for finding the Zariski decomposition, it is clear that $F$ is the fixed
  component of $\arrowvert D \arrowvert$. Now, as in Lem.\! 2.5 and Cor. 2.6
  of \cite{N}, taking global sections of the sequence $0 \to \sI_{C/S}(v)
  \simeq \sO_S(-C+vH) \to \sO_S \to \sO_D \to 0$, we get
  \[ h^1(\sI_C(v)) = h^1(\sI_{C/S}(v)) = h^0( \sO_D) -1 = h^0( \sO_{
    \overline D}) + h^0( \sO_F)-1 = h^0( \sO_F)
\]
for $ \overline D \ne 0$ because $ \overline D \cdot F=0$ ($ h^1(\sI_C(v)) =
h^0( \sO_F)-1$ for $ \overline D =0$). The lines ${E_i}$ are skew and we get
$h^0( \sO_F)= \sum h^0( \sO_{n_iE_i}) = \sum \binom {n_i+1} 2$. Finally
$h^0(\sI_C(v)) = h^0(\sI_{C/S}(v)) + \binom v 3 \ge \binom v 3$ (equality
holds, but we don't need it) and we are done.
\end{proof}

\begin{proof} [Proof of Theorem~\ref{mainC}] \ i) is a special case of
  \cite[Thm. 1]{K1} since $m_3 \ge 3$ easily implies $d > 9$ or that the
  general curve $C$ is a c.i. of type (or bidegree) $(3,3)$.

  ii) By \eqref{new4.6} we get
\begin{equation} 
  \dim W + h^1(\sI_C(3)) = h^0(\sN_C) \ . 
\end{equation}
Since $h^1(\sI_C(3)) \neq 0$, it suffices to prove that $W$ is an irreducible
{\it component\/} of $ \HH(d,g)_{sc}$ because if it is, then $ \dim W <
h^0(\sN_C)$ implies that the general curve $C$ of $W$ is obstructed, i.e. $W$
is non-reduced.

a) To get a contradiction, suppose $W$ is {\it not} a component.
Since $W$ is a maximal family of curves on a cubic surface, there exists a
generization $C'$ of $C$ satisfying $h^0(\sI_{C'}(3)) = 0$. By semi-continuity,
$h^1(\sO_{C'}(4)) \leq h^1(\sO_C(4))$. Combining with $\chi(\sI_{C'}(4)) =
\chi(\sI_C(4))$, it follows that $h^0(\sI_{C'}(4)) - h^1(\sI_{C'}(4)) \geq
h^0(\sI_C(4)) - h^1(\sI_C(4))$. However, by Lemma~\ref{lem1}, we have
$h^0(\sI_C(4)) - h^1(\sI_{C}(4)) \ge 1$, hence $h^0(\sI_{C'}(4)) \geq 1$. Since
the curve is linearly normal by (B), this inequality contradicts the
conclusion of Proposition~\ref{Ellia}.

b) Again it suffices to prove that $W$ is an irreducible
  component of $ \HH(d,g)_{sc}$. To get a contradiction we suppose there is a
generization $C'$ of $C$ satisfying $h^0(\sI_{C'}(3)) = 0$. By
semi-continuity of $h^ 1(\sO_C(v))$ and Lemma~\ref{lem1}, we get  \\[-1mm]

\hspace{1.5cm} $h^0(\sI_{C'}(v)) - h^1(\sI_{C'}(v)) \geq h^0(\sI_C(v)) -
h^1(\sI_C(v)) \ge (_
{3}^ v) - (_{2}^ v)$ for $1 \leq {v} \leq 6$\, . \\[3mm]
Hence $h^ 0(\sI_{C'}(6)) - h^ 1(\sI_{C'}(6)) \geq 5$. Since $s(C') \geq 5$ by
Proposition~\ref{Ellia} and (B), $C'$ is contained in a c.i. of bidegree (5,6)
or (6,6). Hence $d \leq 36$ and we have a contradiction except when $d = 35$
or 36. In the case $d = 36$, $C'$ is a c.i. satisfying $h^0(\sI_{C'}(6)) \geq
5$, and if $d = 35$, we can link $C'$ to a line $D$ satisfying $h^1(\sO_D(2))
\neq 0$ (because $h^0(\sI_{C'}(6)) > 2$), i.e. we get a contradiction in both
cases, and we are done.

c) The proof is similar to b), remarking only that we now have
$h^0(\sI_{C'}(6))\geq 4$ and $h^0(\sI_{C'}(7)) \geq 11$ by Lemma~\ref{lem1},
i.e. $C'$ is contained in a c.i. of bidegree (5,7) or (6,6), and since the
case where $C'$ is a c.i. of bidegree (5,7) can not occur (the dimension of an
irreducible component of $ \HH(d,g)_{sc}$ whose general curve is a c.i. of type
$(5,7)$ is much smaller than $d+g+18$) we conclude as in b).
\end{proof}

\begin{remark} \label{oldPreprint} i) Theorem~\ref{mainC} (without c) of ii))
  was lectured at a workshop organized by the "Space Curves group" of
  Europroj, at the Emile Borel Center, Paris in May 1995, and may be known to
  some experts in the field (cf. \cite[p.\;95]{H2}), but it has not been
  published. The appendix in the preprint \cite{K96} covers the important
  results of the talk, and much of the material is included here. Note that we
  in Lemma\;19 of \cite{K96} should replace equality by inequality, exactly as
  we now do in the displayed formula of Lemma~\ref{lem1} (we see from its
  proof that equality almost always holds, except when $\overline D = 0$).
  This correction does no harm to the arguments of Theorem~\ref{mainC} since it
  is precisely the inequality we need in its proof. In the proof of
  Lemma~\ref{lem1} we follow closely corresponding results in \cite{N} which
  is based on making the fixed component of $\arrowvert C-vH \arrowvert$
  explicit. Lemma~\ref{lem1} for $v=4$ imply Lemma \;18 of \cite{K1}.

  ii) The case a) of Theorem~\ref{mainC} ii) is fully generalized in \cite{N}.
  Indeed Nasu shows that the cup-product (primary obstruction) of the general
  curve of any maximal family $W$ satisfying $m_6=2$ and $m_5 \ge 3$ is
  non-vanishing. We think his approach may be adequate for proving the whole
  conjecture.
\end{remark}

Finally using Proposition~\ref{20} for $s=5$ and closely following the proof
of Theorem~\ref{main4} (replacing $\dim W =g+33$ by $\dim W = d+g+18$ in the
argument and noticing that a generization $C'$ satisfies $s(C')\ge 5$ by Ellia's
Proposition~\ref{Ellia}), we immediately get the following result.
\begin{proposition} \label{4.7} Let $W$ be a 3-maximal family of smooth
  connected space curves, whose general member is linearly normal and sits on
  a smooth cubic surface. If
\begin{equation}  \label{range4.7}
g   >   \max \,\{ \,{\frac{d ^{2} }{10}} - {\frac{d}{2}} +  18  ,\,  
G(d,6)\,\} \, ,  ~ d  \ge   31 \, ,
\end{equation}
then $W$ is an irreducible component of \ $\HH(d,g)_{sc}$. Moreover, $W$ is
non-reduced if and only if $H^ 1(\sI_C(3))\neq 0$. 
In particular Conjecture~\ref{conj} holds in the range \eqref{range4.7}.
\end{proposition}
Note that we in \eqref{range4.7} have $G(d,6) \ge \frac{d ^{2} }{10} -
{\frac{d}{2}} + 18 $ if and only if $ d \le 74$. We can weaken the assumption
$g > G(d,6)$ by using Proposition~\ref{20} also for $s =6$ and 7. Indeed for
any $t$ such that $6 \le t \le 8$ we can conclude as in Proposition~\ref{4.7}
provided
$ g > \max \,\{ \,{\frac{d ^{2} }{10}} - {\frac{d}{2}} + 18 ,\, G(d,t)\,\}
\, , ~ d > t(t-1) \, .$
Since $G(d,8) \le \frac{d ^{2} }{10} - {\frac{d}{2}} + 18 $ for $d \ge 58$, we
obtain all conclusions of Proposition~\ref{4.7} in the range
\begin{equation}  \label{range4.8}
g   >  {\frac{d ^{2} }{10}} - {\frac{d}{2}} +  18  ,  ~ d  \ge   58 \, .
\end{equation}

\bigskip \bigskip

\end{document}